\documentclass[final]{ctextemp_siamltex}
\usepackage{amsfonts}
\usepackage{mathrsfs}
\usepackage{graphicx}
\usepackage{subfigure}
\usepackage{placeins}
\usepackage{float}
\usepackage{color}
\usepackage{comment}
\usepackage{multirow}
\usepackage{multicol}
\usepackage{adjustbox,lipsum}

\usepackage{CJK}
\usepackage{amsmath}

\newtheorem{remark}{\indent Remark}

\begin{document}

\title{Numerical analysis of a bdf2 modular grad-div Stabilization method for the Navier-Stokes equations}

\author{Y. Rong\thanks{Xi'an Jiaotong University, School of Mathematics and Statistics, Xi'an, Shaanxi 710049, China.  Support from NSFC grants 11171269 and 11571274 and China Scholarship Council grant 201606280154.}
\and J. A. Fiordilino\thanks{University of Pittsburgh, Department of Mathematics, Pittsburgh, PA 15260.  The research presented herein was partially supported by NSF grants CBET 1609120 and DMS 1522267.  J.A.F. is supported by the DoD SMART Scholarship.}}

\maketitle

\begin{abstract}
A second-order accurate modular algorithm is presented for a standard BDF2 code for the Navier-Stokes equations (NSE).  
The algorithm exhibits resistance to solver breakdown and increased computational efficiency for increasing values of grad-div parameters.  
We provide a complete theoretical analysis of the algorithms stability and convergency.
Computational tests are performed and illustrate the theory and advantages over monolithic grad-div stabilizations.
\end{abstract}

\section{Introduction}\label{introduction}
A common, powerful tool for improving solution quality for fluid flow problems is grad-div stabilization \cite{Jenkins,Linke2,LeBorne,Olshanskii2, Olshanskii4}.  This technique typically involves adding $\gamma \nabla \nabla \cdot u_{h}$, nonzero for most finite element velocity-pressure pairs, which penalizes mass conservation and improves solution accuracy.
It was first introduced in \cite{Franca} and has been widely studied since, both analytically and computationally
\cite{Decaria,Jenkins,Layton3,Linke2,Lube,Olshanskii,Olshanskii2,Olshanskii4}.

Unfortunately, grad-div stabilization also exhibits increased coupling in the linear system's matrix, efficiency loss and solver breakdown, and classical Poisson locking \cite{Bowers,Guermond,Heister,Linke,LeBorne,Olshanskii,Olshanskii2}. 
In particular, since the matrix arising from grad-div term is singular,
large grad-div parameter $\gamma$ values can cause solver breakdown \cite{Glowinski}.
This difficulty cannot always be circumvented since recommended parameter choices vary greatly, e.g., from $\mathcal{O}(h^{2})$ to $\mathcal{O}(10^{4})$
for different applications, finite elements, and meshes \cite{Decaria,Jenkins,John2,Olshanskii4,Roos}.
An alternate realization of grad-div stabilization with greater computational efficiency was introduced in \cite{Fiordilino} for the backward Euler time discretization.  Herein, we show how to implement modular grad-div stabilization for any multistep time discretization and perform analysis and testing for the BDF2 case.

To begin, consider the incompressible time-dependent NSE: Find the fluid velocity $u:\Omega\times[0,T]\rightarrow\mathbb{R}^{d}$
and pressure $p:\Omega\times(0,T]\rightarrow\mathbb{R}$ satisfying:
\begin{equation}\label{NSE}
\begin{aligned}
&u_{t}+u\cdot\nabla u-\nu\Delta u+\nabla p
= f,\; \mathrm{and} \;
\nabla\cdot u=0\;
\mathrm{in} \;\Omega,\\
&u=0\;
\mathrm{on} \;\partial\Omega,\; \mathrm{and} \;
\int_{\Omega} p\,dx
= 0,\\
&u(x,0)=u^{0}(x)\;
\mathrm{in} \;\Omega.\\
\end{aligned}
\end{equation}
Here, the domain $\Omega\subset \mathbb{R}^{d}$($\mathrm{d}$=2,3) is a bounded polyhedron, $f$ is the body force and $\nu$ is the fluid viscosity.
Suppressing the spacial discretization for the moment, we consider the following two step method that uncouples the grad-div solve.
\\
$\emph{Step 1}:$ Given $u^{n-1}, u^{n}$, find $\hat{u}^{n+1}$ and $p^{n+1}$ satisfying:
\begin{eqnarray}
&\frac{3\hat{u}^{n+1} - 4u^{n} + u^{n-1}}{2\Delta t}
+ (2u^{n}-u^{n-1})\cdot\nabla\hat{u}^{n+1}
- \nu \Delta\hat{u}^{n+1}
+ \nabla p^{n+1}
= f^{n+1},   \label{step01.1}  \\
&\nabla \cdot \hat{u}^{n+1}=0.  \label{step01.2}
\end{eqnarray}
$\emph{Step 2}:$ Given $\hat{u}^{n+1}$, find $u^{n+1}$ satisfying:
\begin{eqnarray}
&\frac{3u^{n+1}-3\hat{u}^{n+1}}{2\Delta t}
- \beta \nabla \nabla \cdot \frac{3u^{n+1} - 4u^{n} + u^{n-1}}{2\Delta t}
- \gamma \nabla \nabla \cdot u^{n+1} = 0.  \label{step02}
\end{eqnarray}
In the above, $\beta \geq 0$ and $ \gamma \geq 0$ are application-dependent grad-div stabilization parameters. The combined effect of $\emph{Step 1}$ and $\emph{Step 2}$
is a consistent BDF2 time discretization of the following model:
\begin{equation}\label{grad-div NSE}
\begin{aligned}
&u_{t} - \beta \nabla \nabla \cdot u_{t} - \gamma \nabla \nabla \cdot u
+ u\cdot\nabla u - \nu\Delta u+\nabla p
= f.
\end{aligned}
\end{equation}
In \cite{Fiordilino}, two minimally intrusive, modular algorithms were developed for backward Euler, which implemented grad-div stabilization.  These algorithms effectively treated issues resulting from increased coupling and solver breakdown.  Although the second steps of each of these algorithms can be used here when $\beta \equiv 0$, they cannot be used when $\beta > 0$; that is, the dispersive term \cite{Decaria2,Layton4,Prohl}, associated with $\beta$ demands special attention.  In the case $\beta > 0$, the time-discretizations in both steps must be consistent with one another.  In particular, for the BDFk family of methods:
\\
$\emph{Step 1}:$ Find $\hat{u}^{n+1}$ and $p^{n+1}$ satisfying:
\begin{eqnarray}
&\frac{1}{\Delta t}\big(a_{0}\hat{u}^{n+1} + \sum^{S}_{s=1} a_{s} u^{n+1-s}\big)
+ U\cdot\nabla\hat{u}^{n+1}
- \nu \Delta\hat{u}^{n+1}
+ \nabla p^{n+1}
= f^{n+1},     \\
&\nabla \cdot \hat{u}^{n+1}=0.  
\end{eqnarray}
$\emph{Step 2}:$ Find $u^{n+1}$ satisfying:
\begin{eqnarray}
&\frac{a_{0}}{\Delta t}\big(u^{n+1}-\hat{u}^{n+1}\big)
- \beta \nabla \nabla \cdot \frac{\sum^{S}_{s=0} a_{s} u^{n+1-s}}{\Delta t}
- \gamma \nabla \nabla \cdot u^{n+1} = 0,  
\end{eqnarray}
where $U$ denotes either $\hat{u}^{n+1}$ or a consistent extrapolation. A similar generalization can be made for general linear multistep methods.

This paper is arranged as follows.
Section $\ref{Preliminaries}$ introduces notation, lemmas, and necessary preliminaries.
In Section $\ref{Algorithm and Stability}$, a fully-discrete modular grad-div stabilization algorithm (\textit{BDF2-mgd}) and its unconditional, nonlinear, energy stability are presented. 
A complete error analysis is given in Section $\ref{Error Analysis}$
where second-order convergence is proven for the modular method.
Numerical experiments are provided to confirm the effectiveness of \textit{BDF2-mgd} in Section $\ref{Numerical Tests}$.  In particular, the algorithm maintains the positive impact of grad-div stabilization while resisting debilitating slow down for $0\leq \gamma \leq 20,000$ or $0 \leq \beta \leq 8,000$.
Conclusions follow in Section $\ref{Conclusion}$.

\section{Preliminaries}\label{Preliminaries}
We use the standard notations $H^{k}(\Omega), H^{k}_{0}(\Omega)$, and $L^{p}(\Omega)$
to denote Sobolev spaces and $L^{p}$ spaces; see, e.g., \cite{Ad}.
The $L^{2}(\Omega)$ inner product and its induced norm are denoted by $(\cdot,\cdot)$ and $\|\cdot\|$, respectively.
Let $\|\cdot\|_{L^{p}}$ and $\|\cdot\|_{k}$ denote the $L^{p}(\Omega)$ ($p\neq2$) norm and $H^{k}(\Omega)$ norm.
The space $H^{-k}(\Omega)$ denotes the dual space of $H^{k}_{0}(\Omega)$ and its norm is denoted by $\|\cdot\|_{-k}$.
Throughout the paper, we use $C$ to denote a generic positive constant varying in different places but never depending on mesh size, time step, and grad-div parameters.
For functions $v(x,t)$,
we define the following norms:
\begin{eqnarray*}
\begin{aligned}
& \|v\|_{\infty,k}:=ess\sup_{[0,T]}\|v(\cdot,t)\|_{k}\, ,\quad
  \|v\|_{p,k}:=(\int^{T}_{0}\|v(\cdot,t)\|^{p}_{k}dt)^{\frac{1}{p}}\,
\end{aligned}
\end{eqnarray*}
for $1\leq p<\infty$.
The velocity space $X$, pressure space $Q$, and divergence free space $V$ are defined as follows.
\begin{eqnarray*}
\begin{aligned}
& X:=H^{1}_{0}(\Omega)^{d}=\{v\in H^{1}(\Omega)^{d}:v|_{\partial\Omega}=0\},\\
& Q:=L^{2}_{0}(\Omega)=\{q\in L^{2}(\Omega):\int_{\Omega}q\,dx=0\},\\
& V:=\{v\in X:(\nabla\cdot v,q)=0\quad\forall q\in Q\}.
\end{aligned}
\end{eqnarray*}
Define the skew-symmetric trilinear form
\begin{eqnarray*}
b(u,v,w):=
\frac{1}{2}(u\cdot\nabla v,w)-\frac{1}{2}(u\cdot\nabla w,v)
\quad \forall\;u,v,w \in X.
\end{eqnarray*}
Then, we have the following estimates for $b$ (see, e.g., Lemma 2.2 in \cite{Layton2}):
\begin{eqnarray}\label{tri}
& b(u,v,w)\leq C\|\nabla u\|\|\nabla v\|\|\nabla w\|,\label{tri1} \\
& b(u,v,w)\leq C\|u\|^{\frac{1}{2}}\|\nabla u\|^{\frac{1}{2}}\|\nabla v\|\|\nabla w\|,\label{tri2}\\
& b(u,v,w)\leq C\|u\|\|v\|_{2}\|\nabla w\|.\label{tri3}
\end{eqnarray}
Divide the simulation time $T$ into $N$ smaller time intervals with $[0,T]=\bigcup\limits_{n=0}^{N-1}[t^{n},t^{n+1}]$,
where $t^{n}=n\Delta t, \, T=N\Delta t$.
We may define the following discrete norms:
\begin{eqnarray*}
\begin{aligned}
& |\|v\||_{\infty,k}:=\max_{0\leq n\leq N}\|v(\cdot,t^{n})\|_{k}\, ,\quad
  |\|v\||_{p,k}:=(\Delta t\sum^{N}_{n=0}\|v(\cdot,t^{n})\|^{p}_{k})^{\frac{1}{p}}\,.
\end{aligned}
\end{eqnarray*}

Let $\Omega_{h}$ be a quasi-uniform mesh of $\Omega$ with $\overline{\Omega}=\bigcup\limits_{K\in\Omega_{h}}K$.
Denote $h=\sup\limits_{K\in\Omega_{h}}diam(K)$.
Let $X_{h}\subset X$ and $Q_{h}\subset Q$ be the finite element spaces.
Assume that $X_{h}$ and $Q_{h}$ satisfy approximation properties of piecewise continuous polynomials on quasi-uniform meshes
of local degrees $k$ and $m$, respectively:
\begin{eqnarray}
& \inf\limits_{v_{h}\in X_{h}}\|u-v_{h}\|
\leq Ch^{k+1}|u|_{k+1}\qquad u\in X\cap H^{k+1}(\Omega)^{d},\\
& \inf\limits_{v_{h}\in X_{h}}\|u-v_{h}\|_{1}
\leq Ch^{k}|u|_{k+1}\qquad\;\;u\in X\cap H^{k+1}(\Omega)^{d},\\
& \inf\limits_{q_{h}\in Q_{h}}\|p-q_{h}\|
\leq Ch^{m+1}|p|_{m+1}\qquad\;\;\; p\in Q\cap H^{m+1}(\Omega).
\end{eqnarray}
Furthermore, we assume that $X_{h}$ and $Q_{h}$ satisfy the usual discrete inf-sup condition:
\begin{eqnarray}\label{inf-sup}
\begin{aligned}
\inf\limits_{q\in Q_{h}}\sup\limits_{v\in X_{h}}\frac{(q,\nabla\cdot v)}{\|\nabla v\|\|q\|}\geq C_{0}>0.
\end{aligned}
\end{eqnarray}
The discrete divergence-free space $V_{h}$ is defined by
\begin{eqnarray*}
\begin{aligned}
V_{h}:=\{v_{h}\in X_{h}:(\nabla\cdot v_{h},q_{h})=0\quad\forall q_{h}\in Q_{h}\}.
\end{aligned}
\end{eqnarray*}
Note that the well-known Taylor-Hood mixed finite element is one such example satisfying the above assumptions with $k=2, m=1$.

The following lemmas will be useful in later analyses.
For their proofs, see Theorem 1.1 on p. 59 of \cite{Girault} for Lemma \ref{lemma0},
Lemma 2 of \cite{RY2} for Lemma \ref{lemma1},
and Lemma 5.1 on p. 369 of \cite{Heywood} for Lemma \ref{gronwall2}.
\begin{lemma}\label{lemma0}
Suppose that the finite element spaces satisfy ($\ref{inf-sup}$). Then, for any $u\in V$, we have
\begin{eqnarray*}
\begin{aligned}
\inf\limits_{v_{h}\in V_{h}}\|\nabla(u-v_{h})\|
\leq C\inf\limits_{v_{h}\in X_{h}}\|\nabla(u-v_{h})\|.
\end{aligned}
\end{eqnarray*}
\end{lemma}

\begin{lemma}\label{lemma1}
If $g_{t}, g_{tt}, g_{ttt}\in L^{2}(0,T;H^{r}(\Omega))$, then we have
\begin{equation}\label{lemma1.1}
\begin{aligned}
\|g^{n+1}-2g^{n}+g^{n-1}\|^{2}_{r}\leq C\Delta t^{3}\int^{t^{n+1}}_{t^{n-1}}\|g_{tt}\|^{2}_{r}dt,
\end{aligned}
\end{equation}
\begin{equation}\label{lemma1.2}
\begin{aligned}
\|3g^{n+1}-4g^{n}+g^{n-1}\|^{2}_{r}\leq C\Delta t\int^{t^{n+1}}_{t^{n-1}}\|g_{t}\|^{2}_{r}dt,
\end{aligned}
\end{equation}
\begin{equation}\label{lemma1.3}
\begin{aligned}
\|\frac{3g^{n+1}-4g^{n}+g^{n-1}}{2\Delta t}-g_{t}(t^{n+1})\|^{2}_{r}
\leq C\Delta t^{3}\int^{t^{n+1}}_{t^{n-1}}\|g_{ttt}\|^{2}_{r}dt.
\end{aligned}
\end{equation}
\end{lemma}

\begin{lemma}{(The discrete Gronwall's lemma,\;without $\Delta t$-restriction)}\label{gronwall2}
Suppose that $n$ and $N$ are nonnegative integers,\;$n\leq N$. The real numbers $a_{n},b_{n},c_{n},\kappa_{n},\Delta t,C$ are nonnegative and satisfy
\begin{eqnarray*}
\begin{aligned}
& a_{N}+\Delta t\sum\limits_{n=0}^{N}b_{n}\leq\Delta t\sum\limits_{n=0}^{N-1}\kappa_{n}a_{n}+\Delta t\sum\limits_{n=0}^{N}c_{n}+C.
\end{aligned}
\end{eqnarray*}
Then,
\begin{eqnarray*}
\begin{aligned}
& a_{N}+\Delta t\sum\limits_{n=0}^{N}b_{n}\leq\exp(\Delta t\sum\limits_{n=0}^{N-1}\kappa_{n})(\Delta t\sum\limits_{n=0}^{N}c_{n}+C).
\end{aligned}
\end{eqnarray*}
\end{lemma}

\section{The BDF2 modular grad-div stabilization algorithm and its stability}\label{Algorithm and Stability}
We propose the following fully-discrete modular grad-div stabilization algorithm for approximating solutions of ($\ref{NSE}$).
\quad \\
\textit{BDF2-mgd}:
\\
$\emph{Step 1}:$ Given $u^{n-1}_{h}, u^{n}_{h} \in X_{h}$,
find $(\hat{u}^{n+1}_{h}, p^{n+1}_{h}) \in (X_{h},Q_{h})$ satisfying:
\begin{eqnarray}
&(\frac{3\hat{u}^{n+1}_{h} - 4u^{n}_{h} + u^{n-1}_{h}}{2\Delta t},v_{h})
+ b(2u^{n}_{h}-u^{n-1}_{h},\hat{u}^{n+1}_{h},v_{h})
+ \nu (\nabla\hat{u}^{n+1}_{h}, \nabla v_{h}) \nonumber  \\
& - (p^{n+1}_{h},\nabla \cdot v_{h})
= (f^{n+1},v_{h}) \quad \forall v_{h} \in X_{h},   \label{step1.1}  \\
&(\nabla \cdot \hat{u}^{n+1}_{h}, q_{h})=0 \quad \forall q_{h} \in Q_{h}.  \label{step1.2}
\end{eqnarray}
$\emph{Step 2}:$ Given $\hat{u}^{n+1}_{h} \in X_{h}$, find $u^{n+1}_{h} \in X_{h}$ satisfying:
\begin{eqnarray}
&(\frac{3u^{n+1}_{h}-3\hat{u}^{n+1}_{h}}{2\Delta t},v_{h})
+ \beta (\nabla \cdot \frac{3u^{n+1}_{h} - 4u^{n}_{h}
+ u^{n-1}_{h}}{2\Delta t}, \nabla \cdot v_{h}) \nonumber \\
&+ \gamma (\nabla \cdot u^{n+1}_{h}, \nabla \cdot v_{h}) = 0 \quad \forall v_{h} \in X_{h}.  \label{step2}
\end{eqnarray}
\begin{remark}
When $\beta = 0$, Step 2 is equivalent to Step 2 appearing in \cite{Fiordilino} with $\gamma \leftarrow \frac{2}{3}\gamma$.
\end{remark}

Step 2 of \textit{BDF2-mgd} appears to be overdetermined since both the tangential and normal components of the solution are prescribed on the boundary.  However, due to the zeroth-order term, it is not; a unique solution always exists, Theorem \ref{existence}, and converges to the true NSE solution, Theorems \ref{error} and \ref{error0}.
\begin{theorem}\label{existence}
Suppose $f^{n+1}\in H^{-1}(\Omega)^{d}$ and $u^{n-1}_{h}, u^{n}_{h} \in X_{h}$.
Then, there exists unique solutions $\hat{u}^{n+1}_{h}, u^{n+1}_{h} \in X_{h}$ to \textit{BDF2-mgd}.
\end{theorem}
\begin{proof}
The proof follows by similar arguments as in Theorem 5 of \cite{Fiordilino}.
\end{proof}

Next, we analyze the stability of \textit{BDF2-mgd}.
We first prove an important lemma for the stability analysis.  Unconditional, nonlinear, energy stability is then proven in Theorem $\ref{stability}$.
\begin{lemma}\label{stability_lemma}
Consider \textit{BDF2-mgd}, then the following identities hold for Step 2 (\ref{step2}):
\begin{eqnarray}\label{stability1}
\begin{aligned}
\|\hat{u}^{n+1}_{h}\|^{2}
=
&\|u^{n+1}_{h}\|^{2} + \|\hat{u}^{n+1}_{h} - u^{n+1}_{h}\|^{2}
+ \frac{4}{3}\gamma \Delta t \|\nabla \cdot u^{n+1}_{h}\|^{2}
\\
&+ \frac{\beta}{3} \Big(\|\nabla \cdot u^{n+1}_{h}\|^{2} - \|\nabla \cdot u^{n}_{h}\|^{2}
+ \|\nabla \cdot (2u^{n+1}_{h}-u^{n}_{h})\|^{2}
\\
&
- \|\nabla \cdot (2u^{n}_{h}-u^{n-1}_{h})\|^{2}
+\| \nabla \cdot (u^{n+1}_{h} - 2u^{n}_{h} + u^{n-1}_{h}) \|^{2} \Big),
\end{aligned}
\end{eqnarray}
and
\begin{eqnarray}\label{stability1.0}
\begin{aligned}
&(\frac{3u^{n+1}_{h} - 4u^{n}_{h} + u^{n-1}_{h}}{2\Delta t},\hat{u}^{n+1}_{h}-u^{n+1}_{h})
\\
&=
\frac{\beta}{6\Delta t}\|\nabla\cdot(3u^{n+1}_{h} - 4u^{n}_{h} + u^{n-1}_{h})\|^{2}
+
\frac{\gamma}{3}(\nabla\cdot u^{n+1}_{h},\nabla\cdot(3u^{n+1}_{h} - 4u^{n}_{h} + u^{n-1}_{h})).
\end{aligned}
\end{eqnarray}
\end{lemma}
\begin{proof}
Selecting $v_{h}=\frac{4\Delta t}{3}u^{n+1}_{h}$ in ($\ref{step2}$), we have
\begin{eqnarray}\label{pf1.1}
\begin{aligned}
&2\|u^{n+1}_{h}\|^{2} - 2(\hat{u}^{n+1}_{h},u^{n+1}_{h})
+ \frac{4}{3}\gamma \Delta t \|\nabla \cdot u^{n+1}_{h}\|^{2}
+ \frac{\beta}{3} \Big(\|\nabla \cdot u^{n+1}_{h}\|^{2}
- \|\nabla \cdot u^{n}_{h}\|^{2}
\\
&+ \|\nabla \cdot (2u^{n+1}_{h}-u^{n}_{h})\|^{2}
- \|\nabla \cdot (2u^{n}_{h}-u^{n-1}_{h})\|^{2}
+\| \nabla \cdot (u^{n+1}_{h} - 2u^{n}_{h} + u^{n-1}_{h}) \|^{2} \Big)
=0,
\end{aligned}
\end{eqnarray}
where we have used the identity $2(3a-4b+c)a=a^{2}-b^{2}+(2a-b)^{2}-(2b-c)^{2}+(a-2b+c)^{2}$ on the third term.
For the second term in (\ref{pf1.1}), using the following polarization identity
\begin{eqnarray}\label{pf1.2}
\begin{aligned}
(\hat{u}^{n+1}_{h},u^{n+1}_{h})
&=\frac{1}{2}\|\hat{u}^{n+1}_{h}\|^{2}
+\frac{1}{2}\|u^{n+1}_{h}\|^{2}
-\frac{1}{2}\|\hat{u}^{n+1}_{h}-u^{n+1}_{h}\|^{2}
\end{aligned}
\end{eqnarray}
yields the first identity (\ref{stability1}).
The second follows by setting $v_{h}=\frac{3u^{n+1}_{h} - 4u^{n}_{h} + u^{n-1}_{h}}{3}$ in ($\ref{step2}$).
\end{proof}

\noindent We are now in a position to prove unconditional stability.

\begin{theorem}\label{stability}
Suppose $f\in L^{2}(0,T;H^{-1}(\Omega)^{d})$, then the following holds for all $N\geq 1$.

\begin{eqnarray}\label{stability2}
\begin{aligned}
&\|u^{N}_{h}\|^{2}
+ \|2u^{N}_{h}-u^{N-1}_{h}\|^{2}
+ (\frac{2\gamma\Delta t}{3} + \beta)\|\nabla\cdot u^{N}_{h}\|^{2}
+ (\frac{2\gamma\Delta t}{3} + \beta)\|\nabla\cdot(2u^{N}_{h}-u^{N-1}_{h})\|^{2}
\\&\quad
+ 4\gamma \Delta t \sum^{N-1}_{n=1}\|\nabla \cdot u^{n+1}_{h}\|^{2}
+ 2\nu \Delta t\sum^{N-1}_{n=1}\|\nabla \hat{u}^{n+1}_{h}\|^{2}
\\
&
\leq
\frac{2\Delta t}{\nu}\sum^{N-1}_{n=1}\|f^{n+1}\|^{2}_{-1}
+ \|u^{1}_{h}\|^{2}
+ \|2u^{1}_{h}-u^{0}_{h}\|^{2}
\\
&\quad
+ (\frac{2\gamma\Delta t}{3} + \beta) \|\nabla \cdot u^{1}_{h}\|^{2}
+ (\frac{2\gamma\Delta t}{3} + \beta) \|\nabla \cdot (2u^{1}_{h}-u^{0}_{h})\|^{2}.
\end{aligned}
\end{eqnarray}
\end{theorem}
\begin{proof}
Set $v_{h}=\hat{u}^{n+1}_{h}$ in ($\ref{step1.1}$) and
$q_{h}=p^{n+1}_{h}$ in ($\ref{step1.2}$).  Adding these two equations and rearranging the discrete time derivative yields
\begin{eqnarray}\label{pf2.1}
\begin{aligned}
&(\frac{3u^{n+1}_{h} - 4u^{n}_{h} + u^{n-1}_{h}}{2\Delta t},u^{n+1}_{h})
+(\frac{3u^{n+1}_{h} - 4u^{n}_{h} + u^{n-1}_{h}}{2\Delta t},\hat{u}^{n+1}_{h}-u^{n+1}_{h})
\\
& +(\frac{3\hat{u}^{n+1}_{h} - 3u^{n+1}_{h}}{2\Delta t},\hat{u}^{n+1}_{h})
+ \nu\|\nabla \hat{u}^{n+1}_{h}\|^{2}
=
(f^{n+1},\hat{u}^{n+1}_{h}).
\end{aligned}
\end{eqnarray}
Consider the resulting time derivative terms.  Use the identity $2(3a-4b+c)a=a^{2}-b^{2}+(2a-b)^{2}-(2b-c)^{2}+(a-2b+c)^{2}$ on the first term and both (\ref{stability1.0}) of Lemma \ref{stability_lemma} and the identity on the second term.  Apply the polarization identity to the third term.  Then,
\begin{eqnarray}\label{pf2.3}
\begin{aligned}
&\frac{1}{4\Delta t}
\Big(\|u^{n+1}_{h}\|^{2} - \|u^{n}_{h}\|^{2}
+ \|2u^{n+1}_{h}-u^{n}_{h}\|^{2} - \|2u^{n}_{h}-u^{n-1}_{h}\|^{2}
+ \|u^{n+1}_{h}-2u^{n}_{h}+u^{n-1}_{h}\|^{2}\Big)
\\&\quad
+ \frac{\gamma}{6}
\Big(\|\nabla\cdot u^{n+1}_{h}\|^{2} - \|\nabla\cdot u^{n}_{h}\|^{2}
+ \|\nabla\cdot(2u^{n+1}_{h}-u^{n}_{h})\|^{2} - \|\nabla\cdot(2u^{n}_{h}-u^{n-1}_{h})\|^{2}
\\&\quad
+ \|\nabla\cdot(u^{n+1}_{h}-2u^{n}_{h}+u^{n-1}_{h})\|^{2}\Big)
+ \frac{\beta}{6\Delta t}\|\nabla\cdot(3u^{n+1}_{h} - 4u^{n}_{h} + u^{n-1}_{h})\|^{2}
\\&\quad
+ \frac{3}{4\Delta t}
\Big(\|\hat{u}^{n+1}_{h}\|^{2} - \|u^{n+1}_{h}\|^{2}
+ \|\hat{u}^{n+1}_{h}-u^{n+1}_{h}\|^{2}\Big)
+ \nu\|\nabla \hat{u}^{n+1}_{h}\|^{2}
\\&
=
(f^{n+1},\hat{u}^{n+1}_{h}).
\end{aligned}
\end{eqnarray}
Multiply ($\ref{pf2.3}$) by $4\Delta t$ and use ($\ref{stability1}$) of Lemma \ref{stability_lemma}.  Then
\begin{eqnarray}\label{pf2.4}
\begin{aligned}
&\|u^{n+1}_{h}\|^{2} - \|u^{n}_{h}\|^{2}
+ \|2u^{n+1}_{h}-u^{n}_{h}\|^{2} - \|2u^{n}_{h}-u^{n-1}_{h}\|^{2}
+ \|u^{n+1}_{h}-2u^{n}_{h}+u^{n-1}_{h}\|^{2}
\\&\quad
+ \frac{2\gamma\Delta t}{3}
\Big(\|\nabla\cdot u^{n+1}_{h}\|^{2} - \|\nabla\cdot u^{n}_{h}\|^{2}
+ \|\nabla\cdot(2u^{n+1}_{h}-u^{n}_{h})\|^{2} - \|\nabla\cdot(2u^{n}_{h}-u^{n-1}_{h})\|^{2}
\\&\quad
+ \|\nabla\cdot(u^{n+1}_{h}-2u^{n}_{h}+u^{n-1}_{h})\|^{2}\Big)
+ \frac{2\beta}{3}\|\nabla\cdot(3u^{n+1}_{h} - 4u^{n}_{h} + u^{n-1}_{h})\|^{2}
\\&\quad
+ \beta \Big(\|\nabla \cdot u^{n+1}_{h}\|^{2} - \|\nabla \cdot u^{n}_{h}\|^{2}
+ \|\nabla \cdot (2u^{n+1}_{h}-u^{n}_{h})\|^{2}
- \|\nabla \cdot (2u^{n}_{h}-u^{n-1}_{h})\|^{2}
\\&\quad
+\| \nabla \cdot (u^{n+1}_{h} - 2u^{n}_{h} + u^{n-1}_{h}) \|^{2} \Big)
+ 6\|\hat{u}^{n+1}_{h} - u^{n+1}_{h}\|^{2}
\\&\quad
+ 4\gamma\Delta t\|\nabla \cdot u^{n+1}_{h}\|^{2}
+ 4\nu\Delta t\|\nabla \hat{u}^{n+1}_{h}\|^{2}
\\&
=
4\Delta t(f^{n+1},\hat{u}^{n+1}_{h}).
\end{aligned}
\end{eqnarray}
Summing ($\ref{pf2.4}$) from $n=1$ to $N-1$ yields
\begin{eqnarray}\label{pf2.5}
\begin{aligned}
&\|u^{N}_{h}\|^{2}
+ \|2u^{N}_{h}-u^{N-1}_{h}\|^{2}
+ \frac{2\gamma\Delta t}{3}\|\nabla\cdot u^{N}_{h}\|^{2}
+ \frac{2\gamma\Delta t}{3}\|\nabla\cdot(2u^{N}_{h}-u^{N-1}_{h})\|^{2}
+ \beta \|\nabla \cdot u^{N}_{h}\|^{2}
\\&\quad
+ \beta \|\nabla \cdot (2u^{N}_{h}-u^{N-1}_{h})\|^{2}
+ 4\gamma \Delta t \sum^{N-1}_{n=1}\|\nabla \cdot u^{n+1}_{h}\|^{2}
+ 4\nu \Delta t\sum^{N-1}_{n=1}\|\nabla \hat{u}^{n+1}_{h}\|^{2}
\\
&
\leq
4\Delta t\sum^{N-1}_{n=1}(f^{n+1},\hat{u}^{n+1}_{h})
+ \|u^{1}_{h}\|^{2}
+ \|2u^{1}_{h}-u^{0}_{h}\|^{2}
\\
&\quad
+ \frac{2\gamma\Delta t}{3}\|\nabla\cdot u^{1}_{h}\|^{2}
+ \frac{2\gamma\Delta t}{3}\|\nabla\cdot(2u^{1}_{h}-u^{0}_{h})\|^{2}
+ \beta \|\nabla \cdot u^{1}_{h}\|^{2}
+ \beta \|\nabla \cdot (2u^{1}_{h}-u^{0}_{h})\|^{2}.
\end{aligned}
\end{eqnarray}
Finally, using the Cauchy-Schwarz-Young inequality on the first term on the right hand side completes the proof.
\end{proof}

\begin{remark}
	Lemma \ref{stability_lemma} and Theorem \ref{stability} imply stability of $ \hat{u}_{h} $ with respect to $ |\|\cdot\||_{\infty, 0} $.
\end{remark}

\section{Error Analysis}\label{Error Analysis}
In this section, we provide $\acute{a}$ priori error estimates for \textit{BDF2-mgd}.  In particular, we show that \textit{BDF2-mgd} is second-order convergent.
Denote $u^{n}=u(t^{n})$ for $n=0,1,\cdots,N$ (and similarly for all other variables).
The errors are denoted by
\begin{eqnarray*}
\begin{aligned}
e_{u}^{n}=u^{n}-u^{n}_{h},\quad
e_{\hat{u}}^{n}=u^{n}-\hat{u}^{n}_{h},\quad
e_{p}^{n}=p^{n}-p^{n+1}_{h}.
\end{aligned}
\end{eqnarray*}
Decompose the velocity errors
\begin{eqnarray*}
\begin{aligned}
&e_{u}^{n}=\eta^{n} - \phi^{n}_{h},\quad
\eta^{n}:=u^{n} - \tilde{u}^{n},\quad
\phi^{n}_{h}:=u^{n}_{h} - \tilde{u}^{n},
\\
&e_{\hat{u}}^{n}=\eta^{n} - \psi^{n}_{h},\quad
\psi^{n}_{h}:=\hat{u}^{n}_{h} - \tilde{u}^{n},
\end{aligned}
\end{eqnarray*}
where $\tilde{u}^{n}$ denotes an interpolant of $u^{n}$ in $V_{h}$.
\begin{definition}
Define the following consistency errors.
For all $v_{h}\in V_{h}$,
\begin{eqnarray}\label{c-errors}
\begin{aligned}
\tau^{n+1}(v_{h}):=
(\frac{3u^{n+1} - 4u^{n} + u^{n-1}}{2\Delta t} - u_{t}^{n+1},v_{h})
- b(u^{n+1}-2u^{n}+u^{n-1},u^{n+1},v_{h}).
\end{aligned}
\end{eqnarray}
\end{definition}

\begin{lemma}\label{c-errors-bound}
Assume the true solution $u$ satisfies the following,
\begin{eqnarray}\label{regularity1}
\begin{aligned}
u\in L^{\infty}(0,T;H^{1}(\Omega)^{d}), \;
u_{tt}\in L^{2}(0,T;H^{1}(\Omega)^{d}), \;
u_{ttt}\in L^{2}(0,T;H^{-1}(\Omega)^{d}).
\end{aligned}
\end{eqnarray}
Then, $\forall \sigma>0$, we have
\begin{eqnarray}
\begin{aligned}
|\tau^{n+1}(v_{h})|
\leq
\frac{C}{2\sigma}\Delta t^{3}\Big(\int^{t^{n+1}}_{t^{n-1}}\|u_{ttt}\|_{-1}^{2}dt
+ \int^{t^{n+1}}_{t^{n-1}}\|\nabla u_{tt}\|^{2}dt\Big)
+ \sigma\|\nabla v_{h}\|^{2}.
\end{aligned}
\end{eqnarray}
\end{lemma}

\begin{proof}
For an arbitrary $\sigma>0$,
\begin{eqnarray}\label{pf3.1}
\begin{aligned}
&|\tau^{n+1}(v_{h})|
\\&\leq
\|\frac{3u^{n+1} - 4u^{n} + u^{n-1}}{2\Delta t} - u_{t}^{n+1}\|_{-1}\|\nabla v_{h}\|
+ C\|\nabla(u^{n+1}-2u^{n}+u^{n-1})\|\|\nabla u^{n+1}\|\|\nabla v_{h}\|
\\&\leq
\frac{1}{2\sigma}\|\frac{3u^{n+1} - 4u^{n} + u^{n-1}}{2\Delta t} - u_{t}^{n+1}\|_{-1}^{2}
\\&\quad
+ \frac{C}{2\sigma}\|\nabla u^{n+1}\|^{2}\|\nabla(u^{n+1}-2u^{n}+u^{n-1})\|^{2}
+ \sigma\|\nabla v_{h}\|^{2}
\\&\leq
\frac{C}{2\sigma}\Delta t^{3} \Big(\int^{t^{n+1}}_{t^{n-1}}\|u_{ttt}\|_{-1}^{2}dt
+ \int^{t^{n+1}}_{t^{n-1}}\|\nabla u_{tt}\|^{2}dt \Big)
+ \sigma\|\nabla v_{h}\|^{2},
\end{aligned}
\end{eqnarray}
where we use the Cauchy-Schwarz-Young inequality and Lemma $\ref{lemma1}$.
\end{proof}

Once again, we require a key lemma, regarding Step 2, to prove convergence.
\begin{lemma}
The following inequality holds.
\begin{eqnarray}\label{erroranalysis1}
\begin{aligned}
\|\psi^{n+1}_{h}\|^{2}
&\geq
\|\phi^{n+1}_{h}\|^{2} + \|\phi^{n+1}_{h}-\psi^{n+1}_{h}\|^{2}
+ \frac{\beta}{3} \Big( \|\nabla\cdot\phi^{n+1}_{h}\|^{2} - \|\nabla\cdot\phi^{n}_{h}\|^{2}
+ \|\nabla\cdot(2\phi^{n+1}_{h}-\phi^{n}_{h})\|^{2}
\\&\quad
- \|\nabla\cdot(2\phi^{n}_{h}-\phi^{n-1}_{h})\|^{2}
+ \frac{1}{2}\|\nabla\cdot(\phi^{n+1}_{h}-2\phi^{n}_{h}+\phi^{n-1}_{h})\|^{2} \Big)
+ \frac{2\gamma\Delta t}{3}\|\nabla\cdot\phi^{n+1}_{h}\|^{2}
\\&\quad
- \frac{C\beta d(1+2\Delta t)}{3}\int^{t^{n+1}}_{t^{n-1}}\|\nabla\eta_{t}\|^{2}dt
- \frac{\beta\Delta t}{3}\|\nabla\cdot(2\phi^{n}_{h}-\phi^{n-1}_{h})\|
- \frac{2\gamma d\Delta t}{3}\|\nabla\eta^{n+1}\|^{2}.
\end{aligned}
\end{eqnarray}
\begin{comment}
and 
\begin{eqnarray}\label{erroranalysis1plus}
\begin{aligned}
\|\psi^{n+1}_{h}\|^{2}
&\leq
\|\phi^{n+1}_{h}\|^{2} + \|\phi^{n+1}_{h}-\psi^{n+1}_{h}\|^{2}
+ \frac{\beta}{3} \Big( \|\nabla\cdot\phi^{n+1}_{h}\|^{2} - \|\nabla\cdot\phi^{n}_{h}\|^{2}
+ \|\nabla\cdot(2\phi^{n+1}_{h}-\phi^{n}_{h})\|^{2}
\\&\quad
- \|\nabla\cdot(2\phi^{n}_{h}-\phi^{n-1}_{h})\|^{2}
+ \|\nabla\cdot(\phi^{n+1}_{h}-2\phi^{n}_{h}+\phi^{n-1}_{h})\|^{2} \Big)
+ 2\gamma\Delta t\|\nabla\cdot\phi^{n+1}_{h}\|^{2}
\\&\quad
+ \frac{C\beta^{2}d}{3\gamma}\int^{t^{n+1}}_{t^{n-1}}\|\nabla\eta_{t}\|^{2}dt
+ \frac{4\gamma d\Delta t}{3}\|\nabla\eta^{n+1}\|^{2}.
\end{aligned}
\end{eqnarray}
\end{comment}
\end{lemma}
\begin{proof}
At time $t^{n+1}$, for all $v_{h} \in X_{h}$, the true solution $u$ satisfies
\begin{eqnarray}\label{pf4.1}
\begin{aligned}
&(\frac{3u^{n+1}-3u^{n+1}}{2\Delta t},v_{h})
+ \beta (\nabla \cdot \frac{3u^{n+1} - 4u^{n}
+ u^{n-1}}{2\Delta t}, \nabla \cdot v_{h})
+ \gamma (\nabla \cdot u^{n+1}, \nabla \cdot v_{h}) = 0.
\end{aligned}
\end{eqnarray}
Subtracting ($\ref{pf4.1}$) from ($\ref{step2}$), we have
\begin{eqnarray}\label{pf4.2}
\begin{aligned}
&(\frac{3e_{u}^{n+1} - 3e_{\hat{u}}^{n+1}}{2\Delta t},v_{h})
+ \beta (\nabla \cdot \frac{3e_{u}^{n+1} - 4e_{u}^{n}
+ e_{u}^{n-1}}{2\Delta t}, \nabla \cdot v_{h})
+ \gamma (\nabla \cdot e_{u}^{n+1}, \nabla \cdot v_{h}) = 0.
\end{aligned}
\end{eqnarray}
Setting $v_{h}=\phi^{n+1}_{h}$ in ($\ref{pf4.2}$), using similar identities as in Theorem \ref{stability_lemma}, and rearranging terms yields
\begin{eqnarray}\label{pf4.3}
\begin{aligned}
\|\psi^{n+1}_{h}\|^{2}
&=
\|\phi^{n+1}_{h}\|^{2} + \|\phi^{n+1}_{h}-\psi^{n+1}_{h}\|^{2}
+ \frac{\beta}{3} \Big( \|\nabla\cdot\phi^{n+1}_{h}\|^{2} - \|\nabla\cdot\phi^{n}_{h}\|^{2}
+ \|\nabla\cdot(2\phi^{n+1}_{h}-\phi^{n}_{h})\|^{2}
\\&\quad
- \|\nabla\cdot(2\phi^{n}_{h}-\phi^{n-1}_{h})\|^{2}
+ \|\nabla\cdot(\phi^{n+1}_{h}-2\phi^{n}_{h}+\phi^{n-1}_{h})\|^{2} \Big)
+ \frac{4\gamma\Delta t}{3}\|\nabla\cdot\phi^{n+1}_{h}\|^{2}
\\&\quad
- \frac{2\beta}{3}(\nabla\cdot(3\eta^{n+1}-4\eta^{n}+\eta^{n-1}),\nabla\cdot\phi^{n+1}_{h})
- \frac{4\gamma\Delta t}{3}(\nabla\cdot\eta^{n+1},\nabla\cdot\phi^{n+1}_{h}).
\end{aligned}
\end{eqnarray}
Split $-\frac{2\beta}{3}(\nabla\cdot(3\eta^{n+1}-4\eta^{n}+\eta^{n-1}),\nabla\cdot\phi^{n+1}_{h})$ into $ - \frac{2\beta}{3}(\nabla\cdot(3\eta^{n+1}-4\eta^{n}+\eta^{n-1}),\nabla \cdot (\phi^{n+1}_{h}-2\phi^{n}_{h}+\phi^{n-1}_{h})) -\frac{2\beta}{3}(\nabla\cdot(3\eta^{n+1}-4\eta^{n}+\eta^{n-1}),\nabla \cdot (2\phi^{n}_{h}-\phi^{n-1}_{h}))$. Using the Cauchy-Schwarz-Young inequality and Lemma $\ref{lemma1}$.  Then, the following three inequalities hold,
\begin{eqnarray}\label{pf4.4}
\begin{aligned}
&|\frac{2\beta}{3}(\nabla\cdot(3\eta^{n+1}-4\eta^{n}+\eta^{n-1}),\nabla\cdot(2\phi^{n}_{h}-\phi^{n-1}_{h}))| 
\\&\leq
\frac{2\beta\sqrt{d}}{3}\|\nabla(3\eta^{n+1}-4\eta^{n}+\eta^{n-1})\|\|\nabla\cdot(2\phi^{n}_{h}-\phi^{n-1}_{h})\|
\\&\leq
\frac{C\beta d}{3}\int^{t^{n+1}}_{t^{n-1}}\|\nabla\eta_{t}\|^{2}dt
+ \frac{\beta\Delta t}{3}\|\nabla\cdot(2\phi^{n}_{h}-\phi^{n-1}_{h})\|^{2}
\\&\leq
\frac{C\beta d}{3}\int^{t^{n+1}}_{t^{n-1}}\|\nabla\eta_{t}\|^{2}dt
+ \frac{\beta\Delta t}{3}\big(\|\nabla\cdot(2\phi^{n}_{h}-\phi^{n-1}_{h})\|^{2} + \|\nabla \cdot \phi^{n}_{h}\|^{2}\big),
\end{aligned}
\end{eqnarray}
\begin{eqnarray}\label{pf4.5}
\begin{aligned}
&|\frac{2\beta}{3}(\nabla\cdot(3\eta^{n+1}-4\eta^{n}+\eta^{n-1}),\nabla\cdot(\phi^{n+1}_{h}-2\phi^{n}_{h}+\phi^{n-1}_{h}))| 
\\&\leq
\frac{2\beta\sqrt{d}}{3}\|\nabla(3\eta^{n+1}-4\eta^{n}+\eta^{n-1})\|\|\nabla\cdot(\phi^{n+1}_{h}-2\phi^{n}_{h}+\phi^{n-1}_{h})\|
\\&\leq
\frac{2C\beta d\Delta t}{3}\int^{t^{n+1}}_{t^{n-1}}\|\nabla\eta_{t}\|^{2}dt
+ \frac{\beta}{6}\|\nabla\cdot(\phi^{n+1}_{h}-2\phi^{n}_{h}+\phi^{n-1}_{h})\|^{2},
\end{aligned}
\end{eqnarray}
and
\begin{eqnarray}\label{pf4.6}
\begin{aligned}
|\frac{4\gamma\Delta t}{3}(\nabla\cdot\eta^{n+1},\nabla\cdot\phi^{n+1}_{h})|
&\leq
\frac{4\gamma\sqrt{d}\Delta t}{3}\|\nabla\eta^{n+1}\|\|\nabla\cdot\phi^{n+1}_{h}\|
\\&\leq
\frac{2\gamma d\Delta t}{3}\|\nabla\eta^{n+1}\|^{2}
+ \frac{2\gamma\Delta t}{3}\|\nabla\cdot\phi^{n+1}_{h}\|.
\end{aligned}
\end{eqnarray}
Combining ($\ref{pf4.3}$) - ($\ref{pf4.6}$) completes the proof.
\end{proof}

Next, we give the main error result for \textit{BDF2-mgd} when $ \beta>0 $.
\begin{theorem}\label{error}
Assume the true solution $u, p$ satisfy ($\ref{regularity1}$) and the following regularity
\begin{eqnarray}\label{regularity2}
\begin{aligned}
&u\in L^{\infty}(0,T;H^{k+1}(\Omega)^{d})\cap L^{2}(0,T;H^{k+1}(\Omega)^{d}), \\
&u_{t}\in L^{2}(0,T;H^{k+1}(\Omega)^{d}), \quad
p\in L^{2}(0,T;H^{m+1}(\Omega)).
\end{aligned}
\end{eqnarray}
Then, we have the following estimates for \textit{BDF2-mgd}.
\begin{eqnarray}\label{erroranalysis2}
\begin{aligned}
&\|e_{u}^{N}\|^{2} + \|2e_{u}^{N}-e_{u}^{N-1}\|^{2}
+ (\frac{2\gamma\Delta t}{3} + \beta)
\Big( \|\nabla\cdot e_{u}^{N}\|^{2} + \|\nabla\cdot(2e_{u}^{N}-e_{u}^{N-1})\|^{2} \Big)
\\&\quad
+ 2\nu \Delta t \sum^{N-1}_{n=1}\|\nabla e_{\hat{u}}^{n+1}\|^{2}
+ 2\gamma\Delta t \sum^{N-1}_{n=1}\|\nabla\cdot e_{u}^{n+1}\|^{2}
\\&\leq
C\exp(C^{*}T)\bigg\{\inf\limits_{v_{h}\in X_{h}}
\Big(  \beta (1+\Delta t) \|\nabla(u - v_{h})_{t}\|_{2,0}^{2}
+ \frac{1}{\nu}\|(u - v_{h})_{t}\|_{2,0}^{2}
\\&\quad
+ (\frac{\gamma^{2}\Delta t}{\beta} + \gamma + \nu + \frac{1}{\nu}) |\|\nabla(u - v_{h})\||_{2,0}^{2}
+ (\frac{2\gamma\Delta t}{3} + \beta + \frac{1}{\nu^{2}}) \|\nabla(u - v_{h})\|_{\infty,0}^{2} 
\\&\quad
+ \|u - v_{h}\|_{\infty,0}^{2}  \Big)
+ \frac{1}{\nu} \inf\limits_{q_{h}\in Q_{h}} |\|p - q_{h}\||_{2,0}^{2}
+ \frac{1}{\nu}\Delta t^{4}
\\&\quad
+ \|e_{u}^{1}\|^{2} + \|2e_{u}^{1}-e_{u}^{0}\|^{2}
+ (\frac{2\gamma\Delta t}{3} + \beta)
\Big( \|\nabla\cdot e_{u}^{1}\|^{2}
+ \|\nabla\cdot(2e_{u}^{1}-e_{u}^{0})\|^{2} \Big)\bigg\}.
\end{aligned}
\end{eqnarray}
\end{theorem}
\begin{proof}
At time $t^{n+1}$, the true solution $u, p$ satisfies
\begin{eqnarray}
&(\frac{3u^{n+1} - 4u^{n} + u^{n-1}}{2\Delta t},v_{h})
+ b(2u^{n}-u^{n-1},u^{n+1},v_{h})
+ \nu (\nabla u^{n+1}, \nabla v_{h}) \nonumber  \\
& - (p^{n+1},\nabla \cdot v_{h})
= (f^{n+1},v_{h}) + \tau^{n+1}(v_{h})  \quad \forall v_{h} \in X_{h},   \label{pf5.1}  \\
&(\nabla \cdot u^{n+1}, q_{h})=0  \quad \forall q_{h} \in Q_{h}.  \label{pf5.2}
\end{eqnarray}
Subtracting ($\ref{step1.1}$) and ($\ref{step1.2}$) from ($\ref{pf5.1}$) and ($\ref{pf5.2}$), respectively, we have
\begin{eqnarray}
&(\frac{3e_{\hat{u}}^{n+1} - 4e_{u}^{n} + e_{u}^{n-1}}{2\Delta t},v_{h})
+ b(2u^{n}-u^{n-1},u^{n+1},v_{h})
- b(2u^{n}_{h}-u^{n-1}_{h},\hat{u}^{n+1}_{h},v_{h}) \nonumber  \\
& + \nu (\nabla e_{\hat{u}}^{n+1}, \nabla v_{h})
- (e_{p}^{n+1},\nabla \cdot v_{h})
= \tau^{n+1}(v_{h})  \quad \forall v_{h} \in X_{h},   \label{pf5.3}  \\
&(\nabla \cdot e_{\hat{u}}^{n+1}, q_{h})=0  \quad \forall q_{h} \in Q_{h}.  \label{pf5.4}
\end{eqnarray}
Set $v_{h}=\psi^{n+1}_{h} \in V_{h}$ in equation ($\ref{pf5.3}$), then
\begin{eqnarray}\label{pf5.5}
\begin{aligned}
&(\frac{3\eta^{n+1} - 4\eta^{n} + \eta^{n-1}}{2\Delta t},\psi^{n+1}_{h})
- (\frac{3\phi^{n+1}_{h} - 4\phi^{n}_{h} + \phi^{n-1}_{h}}{2\Delta t},\phi^{n+1}_{h})
\\&
- (\frac{3\phi^{n+1}_{h} - 4\phi^{n}_{h} + \phi^{n-1}_{h}}{2\Delta t},\psi^{n+1}_{h}-\phi^{n+1}_{h})
- (\frac{3\psi^{n+1}_{h} - 3\phi^{n+1}_{h}}{2\Delta t},\psi^{n+1}_{h})
\\&
+ b(2u^{n}-u^{n-1},u^{n+1},\psi^{n+1}_{h})
- b(2u^{n}_{h}-u^{n-1}_{h},\hat{u}^{n+1}_{h},\psi^{n+1}_{h})
\\&
+ \nu (\nabla\eta^{n+1}, \nabla\psi^{n+1}_{h})
- \nu \|\nabla\psi^{n+1}_{h}\|^{2}
- (p^{n+1} - q_{h},\nabla \cdot \psi^{n+1}_{h})
= \tau^{n+1}(\psi^{n+1}_{h}).
\end{aligned}
\end{eqnarray}
Here, $q_{h}\in Q_{h}$ is arbitrary.
Furthermore, setting $v_{h}=\frac{3\phi_{h}^{n+1} - 4\phi_{h}^{n} + \phi_{h}^{n-1}}{3} \in V_{h}$ in ($\ref{pf4.2}$) and rearranging terms yields
\begin{eqnarray}\label{pf5.6}
\begin{aligned}
&(\frac{3\phi_{h}^{n+1} - 4\phi_{h}^{n} + \phi_{h}^{n-1}}{2\Delta t},\psi^{n+1}_{h}-\phi^{n+1}_{h})
\\&
= \frac{\gamma}{3}(\nabla\cdot(3\phi_{h}^{n+1} - 4\phi_{h}^{n} + \phi_{h}^{n-1}),\nabla\cdot\phi^{n+1}_{h})
+ \frac{\beta}{6\Delta t}\|\nabla\cdot(3\phi_{h}^{n+1} - 4\phi_{h}^{n} + \phi_{h}^{n-1})\|^{2}
\\&\quad
- \frac{\gamma}{3}(\nabla\cdot(3\phi_{h}^{n+1} - 4\phi_{h}^{n} + \phi_{h}^{n-1}),\nabla\cdot\eta^{n+1})
\\&\quad
- \frac{\beta}{6\Delta t}(\nabla\cdot(3\eta^{n+1} - 4\eta^{n} + \eta^{n-1}),\nabla\cdot(3\phi_{h}^{n+1} - 4\phi_{h}^{n} + \phi_{h}^{n-1})).
\end{aligned}
\end{eqnarray}
Combine ($\ref{pf5.5}$) and ($\ref{pf5.6}$) and rearrange.  Then,
\begin{eqnarray}\label{pf5.7}
\begin{aligned}
&(\frac{3\phi^{n+1}_{h} - 4\phi^{n}_{h} + \phi^{n-1}_{h}}{2\Delta t},\phi^{n+1}_{h})
+ \frac{\gamma}{3}(\nabla\cdot(3\phi_{h}^{n+1} - 4\phi_{h}^{n} + \phi_{h}^{n-1}),\nabla\cdot\phi^{n+1})
\\&\quad
+ \frac{\beta}{6\Delta t}\|\nabla\cdot(3\phi_{h}^{n+1} - 4\phi_{h}^{n} + \phi_{h}^{n-1})\|^{2}
+ \nu \|\nabla\psi^{n+1}_{h}\|^{2}
\\&\quad
+ \frac{3}{4\Delta t}(\|\psi^{n+1}_{h}\|^{2} - \|\phi^{n+1}_{h}\|^{2} + \|\psi^{n+1}_{h}-\phi^{n+1}_{h}\|^{2})
\\&=
(\frac{3\eta^{n+1} - 4\eta^{n} + \eta^{n-1}}{2\Delta t},\psi^{n+1}_{h})
+ \frac{\gamma}{3}(\nabla\cdot(3\phi_{h}^{n+1} - 4\phi_{h}^{n} + \phi_{h}^{n-1}),\nabla\cdot\eta^{n+1})
\\&\quad
+ \frac{\beta}{6\Delta t}(\nabla\cdot(3\eta^{n+1} - 4\eta^{n} + \eta^{n-1}),\nabla\cdot(3\phi_{h}^{n+1} - 4\phi_{h}^{n} + \phi_{h}^{n-1}))
\\&\quad
+ b(2u^{n}-u^{n-1},u^{n+1},\psi^{n+1}_{h})
- b(2u^{n}_{h}-u^{n-1}_{h},\hat{u}^{n+1}_{h},\psi^{n+1}_{h})
\\&\quad
+ \nu (\nabla\eta^{n+1}, \nabla\psi^{n+1}_{h})
- (p^{n+1} - q_{h},\nabla \cdot \psi^{n+1}_{h})
- \tau^{n+1}(\psi^{n+1}_{h}).
\end{aligned}
\end{eqnarray}
Multiplying ($\ref{pf5.7}$) by $4\Delta t$ and use ($\ref{erroranalysis1}$).  Then,
\begin{eqnarray}\label{pf5.8}
\begin{aligned}
&\|\phi^{n+1}_{h}\|^{2} - \|\phi^{n}_{h}\|^{2}
+ \|2\phi^{n+1}_{h}-\phi^{n}_{h}\|^{2}
- \|2\phi^{n}_{h}-\phi^{n-1}_{h}\|^{2}
+ \|\phi^{n+1}_{h}-2\phi^{n}_{h}+\phi^{n-1}_{h}\|^{2}
\\&\quad
+ (\frac{2\gamma\Delta t}{3} + \beta) \Big( \|\nabla\cdot\phi^{n+1}_{h}\|^{2} - \|\nabla\cdot\phi^{n}_{h}\|^{2}
+ \|\nabla\cdot(2\phi^{n+1}_{h}-\phi^{n}_{h})\|^{2}
\\&\quad
- \|\nabla\cdot(2\phi^{n}_{h}-\phi^{n-1}_{h})\|^{2}
+ \|\nabla\cdot(\phi^{n+1}_{h}-2\phi^{n}_{h}+\phi^{n-1}_{h})\|^{2} \Big)
\\&\quad
+ \frac{2\beta}{3}\|\nabla\cdot(3\phi_{h}^{n+1} - 4\phi_{h}^{n} + \phi_{h}^{n-1})\|^{2}
+ 6 \|\psi^{n+1}_{h}-\phi^{n+1}_{h}\|^{2}
\\&\quad
+ 4\nu \Delta t \|\nabla\psi^{n+1}_{h}\|^{2}
+ 2\gamma\Delta t\|\nabla\cdot\phi^{n+1}_{h}\|^{2}
\\&\leq
2(3\eta^{n+1} - 4\eta^{n} + \eta^{n-1},\psi^{n+1}_{h})
+ \frac{4\gamma\Delta t}{3}(\nabla\cdot(3\phi_{h}^{n+1} - 4\phi_{h}^{n} + \phi_{h}^{n-1}),\nabla\cdot\eta^{n+1})
\\&\quad
+ \frac{2\beta}{3}(\nabla\cdot(3\eta^{n+1} - 4\eta^{n} + \eta^{n-1}),\nabla\cdot(3\phi_{h}^{n+1} - 4\phi_{h}^{n} + \phi_{h}^{n-1}))
\\&\quad
+ 4\Delta t b(2u^{n}-u^{n-1},u^{n+1},\psi^{n+1}_{h})
- 4\Delta t b(2u^{n}_{h}-u^{n-1}_{h},\hat{u}^{n+1}_{h},\psi^{n+1}_{h})
\\&\quad
+ 4\nu\Delta t (\nabla\eta^{n+1}, \nabla\psi^{n+1}_{h})
- 4\Delta t(p^{n+1} - q_{h},\nabla \cdot \psi^{n+1}_{h})
- 4\Delta t\tau^{n+1}(\psi^{n+1}_{h})
\\&\quad
+ C\beta d(1+2 \Delta t)\int^{t^{n+1}}_{t^{n-1}}\|\nabla\eta_{t}\|^{2}dt
+ 2\gamma d\Delta t\|\nabla\eta^{n+1}\|^{2}
\\&\quad
+ \frac{\beta}{2} \|\nabla\cdot(\phi^{n+1}_{h}-2\phi^{n+1}_{h}+\phi^{n}_{h})\|^{2}
+ \beta \Delta t \big( \|\nabla\cdot(2\phi^{n}_{h}-\phi^{n-1}_{h})\|^{2}
+ \|\nabla\cdot\phi^{n}_{h}\|^{2} \big).
\end{aligned}
\end{eqnarray}

Next, we need to bound the terms on the right hand side of ($\ref{pf5.8}$).
Applying Lemma $\ref{lemma1}$, the Poincar\'{e}-Friedrichs inequality,
and the Cauchy-Schwarz-Young inequality, for an arbitrary $\delta>0$, we have
\begin{eqnarray}\label{pf5.9}
\begin{aligned}
2(3\eta^{n+1} - 4\eta^{n} + \eta^{n-1},\psi^{n+1}_{h})
&\leq
C\|3\eta^{n+1} - 4\eta^{n} + \eta^{n-1}\|\|\nabla\psi^{n+1}_{h}\|
\\&\leq
\frac{C}{\delta\nu}\int^{t^{n+1}}_{t^{n-1}}\|\eta_{t}\|^{2}dt + \delta\nu \Delta t \|\nabla\psi^{n+1}_{h}\|^{2}.
\end{aligned}
\end{eqnarray}
\begin{eqnarray}\label{pf5.10}
\begin{aligned}
&\frac{4\gamma\Delta t}{3}(\nabla\cdot(3\phi_{h}^{n+1} - 4\phi_{h}^{n} + \phi_{h}^{n-1}),\nabla\cdot\eta^{n+1})
\\&\leq
\frac{4\gamma\sqrt{d}\Delta t}{3}\|\nabla\cdot(3\phi_{h}^{n+1} - 4\phi_{h}^{n} + \phi_{h}^{n-1})\|\|\nabla\eta^{n+1}\|
\\&\leq
\frac{\beta}{3}\|\nabla\cdot(3\phi_{h}^{n+1} - 4\phi_{h}^{n} + \phi_{h}^{n-1})\|^{2}
+ \frac{4 d\gamma^{2}\Delta t^{2}}{3\beta}\|\nabla\eta^{n+1}\|^{2}.
\end{aligned}
\end{eqnarray}
\begin{eqnarray}\label{pf5.11}
\begin{aligned}
&\frac{2\beta}{3}(\nabla\cdot(3\eta^{n+1} - 4\eta^{n} + \eta^{n-1}),\nabla\cdot(3\phi_{h}^{n+1} - 4\phi_{h}^{n} + \phi_{h}^{n-1}))
\\&\leq
\frac{2 \beta\sqrt{d}}{3}
\|\nabla(3\eta^{n+1} - 4\eta^{n} + \eta^{n-1})\|\|\nabla\cdot(3\phi_{h}^{n+1} - 4\phi_{h}^{n} + \phi_{h}^{n-1})\|
\\&\leq
\frac{\beta}{3}\|\nabla\cdot(3\phi_{h}^{n+1} - 4\phi_{h}^{n} + \phi_{h}^{n-1})\|^{2}
+ \frac{C\beta d \Delta t}{3}\int^{t^{n+1}}_{t^{n-1}}\|\nabla\eta_{t}\|^{2}dt.
\end{aligned}
\end{eqnarray}
Furthermore,
\begin{eqnarray}\label{pf5.12}
\begin{aligned}
&4\nu\Delta t (\nabla\eta^{n+1}, \nabla\psi^{n+1}_{h})
\leq
\frac{4\nu\Delta t}{\delta}\|\nabla\eta^{n+1}\|^{2} + \delta\nu \Delta t \|\nabla\psi^{n+1}_{h}\|^{2}.
\end{aligned}
\end{eqnarray}
\begin{eqnarray}\label{pf5.13}
\begin{aligned}
&- 4\Delta t(p^{n+1} - q_{h},\nabla \cdot \psi^{n+1}_{h})
\leq
\frac{4d\Delta t}{\delta\nu}\|p^{n+1} - q_{h}\|^{2} + \delta\nu \Delta t \|\nabla\psi^{n+1}_{h}\|^{2}.
\end{aligned}
\end{eqnarray}
Applying Lemma $\ref{c-errors-bound}$ yields
\begin{eqnarray}\label{pf5.14}
\begin{aligned}
&- 4\Delta t\tau^{n+1}(\psi^{n+1}_{h})
\\&\leq
\frac{C\Delta t^{4}}{\delta\nu}\int^{t^{n+1}}_{t^{n-1}}\|u_{ttt}\|_{-1}^{2}dt
+ \frac{C\Delta t^{4}}{\delta\nu}\int^{t^{n+1}}_{t^{n-1}}\|\nabla u_{tt}\|^{2}dt
+ \delta\nu\Delta t\|\nabla\psi^{n+1}_{h}\|^{2}.
\end{aligned}
\end{eqnarray}
For the nonlinear terms, we treat them as follows.  Adding and subtracting $4\Delta t b(2u^{n}_{h}-u^{n-1}_{h},u^{n+1},\psi^{n+1}_{h})$ yields
\begin{eqnarray}\label{pf5.15}
\begin{aligned}
&4\Delta t b(2u^{n}-u^{n-1},u^{n+1},\psi^{n+1}_{h})
- 4\Delta t b(2u^{n}_{h}-u^{n-1}_{h},\hat{u}^{n+1}_{h},\psi^{n+1}_{h})
\\&=
4\Delta t \Big( b(2\eta^{n}-\eta^{n-1},u^{n+1},\psi^{n+1}_{h})
- b(2\phi^{n}_{h}-\phi^{n-1}_{h},u^{n+1},\psi^{n+1}_{h})
\\&\quad
+ b(2\hat{u}_{h}^{n}-\hat{u}_{h}^{n-1},\eta^{n+1},\psi^{n+1}_{h}) \Big).
\end{aligned}
\end{eqnarray}
Then,
\begin{eqnarray}\label{pf5.16}
\begin{aligned}
&4\Delta t b(2\eta^{n}-\eta^{n-1},u^{n+1},\psi^{n+1}_{h})
\leq
4C\Delta t \|\nabla(2\eta^{n}-\eta^{n-1})\|\|\nabla u^{n+1}\|\|\nabla\psi^{n+1}_{h}\|
\\&\leq
\frac{4C\Delta t}{\delta\nu} \|\nabla(2\eta^{n}-\eta^{n-1})\|^{2}\|\nabla u^{n+1}\|^{2}
+ \delta\nu\Delta t\|\nabla\psi^{n+1}_{h}\|^{2}
\\&\leq
\frac{16C\Delta t}{\delta\nu} ( \|\nabla\eta^{n}\|^{2} + \|\nabla\eta^{n-1}\|^{2} ) \|\nabla u\|_{\infty,0}^{2}
+ \delta\nu\Delta t\|\nabla\psi^{n+1}_{h}\|^{2},
\end{aligned}
\end{eqnarray}
\begin{eqnarray}\label{pf5.17}
\begin{aligned}
&-4\Delta t b(2\phi^{n}_{h}-\phi^{n-1}_{h},u^{n+1},\psi^{n+1}_{h})
\leq
4C\Delta t \|2\phi^{n}_{h}-\phi^{n-1}_{h}\|\|u^{n+1}\|_{2}\|\nabla\psi^{n+1}_{h}\|
\\&\leq
\frac{4C\Delta t}{\delta\nu} \|2\phi^{n}_{h}-\phi^{n-1}_{h}\|^{2}\|u^{n+1}\|_{2}^{2}
+ \delta\nu\Delta t\|\nabla\psi^{n+1}_{h}\|^{2}
\\&\leq
\frac{8C\Delta t}{\delta\nu} ( \|2\phi^{n}_{h}-\phi^{n-1}_{h}\|^{2} + \|\phi^{n}_{h}\|^{2}) \|u^{n+1}\|_{2}^{2}
+ \delta\nu\Delta t\|\nabla\psi^{n+1}_{h}\|^{2},
\end{aligned}
\end{eqnarray}
\begin{eqnarray}\label{pf5.18}
\begin{aligned}
&4\Delta t b(2\hat{u}_{h}^{n}-\hat{u}_{h}^{n-1},\eta^{n+1},\psi^{n+1}_{h})
\leq
4C\Delta t \|\nabla(2\hat{u}_{h}^{n}-\hat{u}_{h}^{n-1})\|\|\nabla\eta^{n+1}\|\|\nabla\psi^{n+1}_{h}\|
\\&\leq
\frac{4C\Delta t}{\delta\nu} \|\nabla(2\hat{u}_{h}^{n}-\hat{u}_{h}^{n-1})\|^{2}\|\nabla\eta^{n+1}\|^{2}
+ \delta\nu\Delta t\|\nabla\psi^{n+1}_{h}\|^{2}
\\&\leq
\frac{16C\Delta t}{\delta\nu} ( \|\nabla\hat{u}_{h}^{n}\|^{2} + \|\nabla\hat{u}_{h}^{n-1}\|^{2} ) \|\nabla\eta\|_{\infty,0}^{2}
+ \delta\nu\Delta t\|\nabla\psi^{n+1}_{h}\|^{2}.
\end{aligned}
\end{eqnarray}
Setting $\delta= \frac{2}{7}$ and using the estimates ($\ref{pf5.9}$)-($\ref{pf5.18}$) in ($\ref{pf5.8}$) yields
\begin{eqnarray}\label{pf5.19}
\begin{aligned}
&\|\phi^{n+1}_{h}\|^{2} - \|\phi^{n}_{h}\|^{2}
+ \|2\phi^{n+1}_{h}-\phi^{n}_{h}\|^{2}
- \|2\phi^{n}_{h}-\phi^{n-1}_{h}\|^{2}
+ \|\phi^{n+1}_{h}-2\phi^{n}_{h}+\phi^{n-1}_{h}\|^{2}
\\&\quad
+ (\frac{2\gamma\Delta t}{3} + \beta) \Big( \|\nabla\cdot\phi^{n+1}_{h}\|^{2} - \|\nabla\cdot\phi^{n}_{h}\|^{2}
+ \|\nabla\cdot(2\phi^{n+1}_{h}-\phi^{n}_{h})\|^{2}
- \|\nabla\cdot(2\phi^{n}_{h}-\phi^{n-1}_{h})\|^{2}
\\&\quad
+ \frac{1}{2}\|\nabla\cdot(\phi^{n+1}_{h}-2\phi^{n}_{h}+\phi^{n-1}_{h})\|^{2} \Big)
+ 6 \|\psi^{n+1}_{h}-\phi^{n+1}_{h}\|^{2}
+ 2\nu \Delta t \|\nabla\psi^{n+1}_{h}\|^{2}
+ 2\gamma\Delta t\|\nabla\cdot\phi^{n+1}_{h}\|^{2}
\\&\leq
\frac{C\Delta t}{\nu} \|u^{n+1}\|_{2}^{2} ( \|2\phi^{n}_{h}-\phi^{n-1}_{h}\|^{2} + \|\phi^{n}_{h}\|^{2}) + \beta \Delta t (\|\nabla\cdot(2\phi^{n}_{h}-\phi^{n-1}_{h})\|^{2} + \|\nabla\cdot\phi^{n}_{h}\|^{2})
\\&\quad
+ {C\beta d \Delta t}\int^{t^{n+1}}_{t^{n-1}}\|\nabla\eta_{t}\|^{2}dt
+ \frac{C}{\nu}\int^{t^{n+1}}_{t^{n-1}}\|\eta_{t}\|^{2}dt
+ C(\frac{d\gamma^{2}\Delta t}{\beta} + \nu)\Delta t\|\nabla\eta^{n+1}\|^{2}
\\&\quad
+ \frac{Cd\Delta t}{\nu}\|p^{n+1} - q_{h}\|^{2}
+ \frac{C\Delta t^{4}}{\nu}\int^{t^{n+1}}_{t^{n-1}}\|u_{ttt}\|_{-1}^{2}dt
+ \frac{C\Delta t^{4}}{\nu}\int^{t^{n+1}}_{t^{n-1}}\|\nabla u_{tt}\|^{2}dt
\\&\quad
+ \frac{C\Delta t}{\nu} ( \|\nabla\eta^{n}\|^{2} + \|\nabla\eta^{n-1}\|^{2} )
+ \frac{C\Delta t}{\nu} ( \|\nabla\hat{u}_{h}^{n}\|^{2} + \|\nabla\hat{u}_{h}^{n-1}\|^{2} ) \|\nabla\eta\|_{\infty,0}^{2}.
\end{aligned}
\end{eqnarray}
Sum ($\ref{pf5.19}$) from $n=1$ to $N-1$ to get
\begin{eqnarray}\label{pf5.20}
\begin{aligned}
&\|\phi^{N}_{h}\|^{2} + \|2\phi^{N}_{h}-\phi^{N-1}_{h}\|^{2}
+ (\frac{2\gamma\Delta t}{3} + \beta)
\Big( \|\nabla\cdot\phi^{N}_{h}\|^{2} + \|\nabla\cdot(2\phi^{N}_{h}-\phi^{N-1}_{h})\|^{2} \Big)
\\&\quad
+ \frac{1}{2}\sum^{N-1}_{n=1}\|\nabla\cdot(\phi^{n+1}_{h}-2\phi^{n}_{h}+\phi^{n-1}_{h})\|^{2}
+ 2\nu \Delta t \sum^{N-1}_{n=1}\|\nabla\psi^{n+1}_{h}\|^{2}
+ 2\gamma\Delta t \sum^{N-1}_{n=1}\|\nabla\cdot\phi^{n+1}_{h}\|^{2}
\\&\leq
\Delta t\sum^{N-1}_{n=1} \Big(C\nu^{-1}\|u^{n+1}\|_{2}^{2} ( \|\phi^{n}_{h}\|^{2} + \|2\phi^{n}_{h}-\phi^{n-1}_{h}\|^{2} ) + \beta (\|\nabla\cdot\phi^{n}_{h}\|^{2} + \|\nabla\cdot(2\phi^{n}_{h}-\phi^{n-1}_{h})\|^{2}) \Big)
\\&\quad
+ \frac{C\beta d \Delta t}{3}\|\nabla\eta_{t}\|_{2,0}^{2}
+ \frac{C}{\nu}\|\eta_{t}\|_{2,0}^{2}
+ C(\frac{d\gamma^{2}\Delta t}{\beta} + \nu + \frac{1}{\nu}) |\|\nabla\eta\||_{2,0}^{2}
\\&\quad
+ \frac{Cd}{\nu}  |\|p - q_{h}\||_{2,0}^{2}
+ \frac{C\Delta t^{4}}{\nu}\|u_{ttt}\|_{2,-1}^{2}
+ \frac{C\Delta t^{4}}{\nu}\|\nabla u_{tt}\|_{2,0}^{2}
+ \frac{C}{\nu^{2}}( \nu\Delta t\sum^{N-1}_{n=0} \|\nabla\hat{u}_{h}^{n}\|^{2} ) \|\nabla\eta\|_{\infty,0}^{2}
\\&\quad
+ \|\phi^{1}_{h}\|^{2} + \|2\phi^{1}_{h}-\phi^{0}_{h}\|^{2}
+ (\frac{2\gamma\Delta t}{3} + \beta)
\Big( \|\nabla\cdot\phi^{1}_{h}\|^{2} + \|\nabla\cdot(2\phi^{1}_{h}-\phi^{0}_{h})\|^{2} \Big).
\end{aligned}
\end{eqnarray}
Denote $C^{*}=\max\{\frac{C}{\nu}|\|u\||_{2,2}^{2},1\}$.
Then, Lemma \ref{gronwall2}, the boundedness of $\nu\Delta t \sum\limits^{N-1}_{n=1} \|\nabla\hat{u}_{h}^{n+1}\|^{2}$ (Theorem $\ref{stability}$), and taking infimums over $V_{h}$ and $Q_{h}$ yield
\begin{eqnarray}\label{pf5.21}
\begin{aligned}
&\|\phi^{N}_{h}\|^{2} + \|2\phi^{N}_{h}-\phi^{N-1}_{h}\|^{2}
+ (\frac{2\gamma\Delta t}{3} + \beta)
\Big( \|\nabla\cdot\phi^{N}_{h}\|^{2} + \|\nabla\cdot(2\phi^{N}_{h}-\phi^{N-1}_{h})\|^{2} \Big)
\\&\quad
+ \frac{1}{2}\sum^{N-1}_{n=1}\|\nabla\cdot(\phi^{n+1}_{h}-2\phi^{n}_{h}+\phi^{n-1}_{h})\|^{2}
+ 2\nu \Delta t \sum^{N-1}_{n=1}\|\nabla\psi^{n+1}_{h}\|^{2}
+ 2\gamma\Delta t \sum^{N-1}_{n=1}\|\nabla\cdot\phi^{n+1}_{h}\|^{2}
\\&\leq
C\exp(C^{*}T)\bigg\{\inf\limits_{v_{h}\in V_{h}}
\Big(  \beta(1+\Delta t) \|\nabla(u - v_{h})_{t}\|_{2,0}^{2}
+ \frac{1}{\nu}\|(u - v_{h})_{t}\|_{2,0}^{2}
\\&\quad
+ (\frac{\gamma^{2}\Delta t}{\beta} + \nu + \frac{1}{\nu}) |\|\nabla(u - v_{h})\||_{2,0}^{2}
+ \frac{1}{\nu^{2}} \|\nabla(u - v_{h})\|_{\infty,0}^{2} \Big)
\\&\quad
+ \frac{1}{\nu} \inf\limits_{q_{h}\in Q_{h}} |\|p - q_{h}\||_{2,0}^{2}
+ \frac{1}{\nu}\Delta t^{4}
\\&\quad
+ \|\phi^{1}_{h}\|^{2} + \|2\phi^{1}_{h}-\phi^{0}_{h}\|^{2}
+ (\frac{2\gamma\Delta t}{3} + \beta)
\Big( \|\nabla\cdot\phi^{1}_{h}\|^{2} + \|\nabla\cdot(2\phi^{1}_{h}-\phi^{0}_{h})\|^{2} \Big)\bigg\}.
\end{aligned}
\end{eqnarray}
Then, using Lemma \ref{lemma0} and the triangle inequality completes the proof.
\end{proof}

The above result has dependence on $\beta^{-1}$.  Consequently, we consider the convergency of \textit{BDF2-mgd}  when $ \beta=0 $ separately.
\begin{theorem}\label{error0}
	Assume the true solution $u, p$ satisfy ($\ref{regularity1}$) and ($\ref{regularity2}$). 
	Then, when $ \beta=0 $, we have the following estimates for \textit{BDF2-mgd}.
	\begin{eqnarray}\label{erroranalysis0}
	\begin{aligned}
	&\|e_{u}^{N}\|^{2} + \|2e_{u}^{N}-e_{u}^{N-1}\|^{2}
	+ \frac{2\gamma\Delta t}{3}
	\Big( \|\nabla\cdot e_{u}^{N}\|^{2} + \|\nabla\cdot(2e_{u}^{N}-e_{u}^{N-1})\|^{2} \Big)
	\\&\quad
	+ 2\nu \Delta t \sum^{N-1}_{n=1}\|\nabla e_{\hat{u}}^{n+1}\|^{2}
	+ 2\gamma\Delta t \sum^{N-1}_{n=1}\|\nabla\cdot e_{u}^{n+1}\|^{2}
	\\&\leq
	C\exp(C^{**}T)
	\bigg\{\inf\limits_{v_{h}\in X_{h}}
	\Big( \frac{1}{\nu}\|(u - v_{h})_{t}\|_{2,0}^{2}
	+ \|u - v_{h}\|_{\infty,0}^{2}	
	+ ( \gamma + \nu + \frac{1}{\nu}) |\|\nabla(u - v_{h})\||_{2,0}^{2}
	\\&\quad
	+ (\frac{2\gamma\Delta t}{3} + \frac{1}{\nu^{2}}) \|\nabla(u - v_{h})\|_{\infty,0}^{2} \Big)
	+ \frac{1}{\nu} \inf\limits_{q_{h}\in Q_{h}} |\|p - q_{h}\||_{2,0}^{2}
	+ \frac{1}{\nu}\Delta t^{4}
	\\&\quad
	+ \|e_{u}^{1}\|^{2} + \|2e_{u}^{1}-e_{u}^{0}\|^{2}
	+ \gamma\Delta t
	 \|\nabla\cdot e_{u}^{1}\|^{2}
	+ \gamma\Delta t \|\nabla\cdot(2e_{u}^{1}-e_{u}^{0})\|^{2} \bigg\}.
	\end{aligned}
	\end{eqnarray}
\end{theorem}
\begin{proof}
Similar to ($ \ref{pf5.8} $), we have 
\begin{eqnarray}\label{pf6.1}
\begin{aligned}
&\|\phi^{n+1}_{h}\|^{2} - \|\phi^{n}_{h}\|^{2}
+ \|2\phi^{n+1}_{h}-\phi^{n}_{h}\|^{2}
- \|2\phi^{n}_{h}-\phi^{n-1}_{h}\|^{2}
+ \|\phi^{n+1}_{h}-2\phi^{n}_{h}+\phi^{n-1}_{h}\|^{2}
\\&\quad
+ \frac{2\gamma\Delta t}{3}  \Big( \|\nabla\cdot\phi^{n+1}_{h}\|^{2} - \|\nabla\cdot\phi^{n}_{h}\|^{2}
+ \|\nabla\cdot(2\phi^{n+1}_{h}-\phi^{n}_{h})\|^{2}
\\&\quad
- \|\nabla\cdot(2\phi^{n}_{h}-\phi^{n-1}_{h})\|^{2}
+ \|\nabla\cdot(\phi^{n+1}_{h}-2\phi^{n}_{h}+\phi^{n-1}_{h})\|^{2} \Big)
\\&\quad
+ 6 \|\psi^{n+1}_{h}-\phi^{n+1}_{h}\|^{2}
+ 4\nu \Delta t \|\nabla\psi^{n+1}_{h}\|^{2}
+ 2\gamma\Delta t\|\nabla\cdot\phi^{n+1}_{h}\|^{2}
\\&\leq
2(3\eta^{n+1} - 4\eta^{n} + \eta^{n-1},\psi^{n+1}_{h})
+ \frac{4\gamma\Delta t}{3}(\nabla\cdot(3\phi_{h}^{n+1} - 4\phi_{h}^{n} + \phi_{h}^{n-1}),\nabla\cdot\eta^{n+1})
\\&\quad
+ 4\Delta t b(2u^{n}-u^{n-1},u^{n+1},\psi^{n+1}_{h})
- 4\Delta t b(2u^{n}_{h}-u^{n-1}_{h},\hat{u}^{n+1}_{h},\psi^{n+1}_{h})
\\&\quad
+ 4\nu\Delta t (\nabla\eta^{n+1}, \nabla\psi^{n+1}_{h})
- 4\Delta t(p^{n+1} - q_{h},\nabla \cdot \psi^{n+1}_{h})
- 4\Delta t\tau^{n+1}(\psi^{n+1}_{h})
\\&\quad
+ \frac{4\gamma d\Delta t}{3}\|\nabla\eta^{n+1}\|^{2}.
\end{aligned}
\end{eqnarray}
Since $ \beta=0 $, we estimate $\frac{4\gamma\Delta t}{3}(\nabla\cdot(3\phi_{h}^{n+1} - 4\phi_{h}^{n} + \phi_{h}^{n-1}),\nabla\cdot\eta^{n+1}) $ as follows,
\begin{eqnarray}\label{pf6.2}
\begin{aligned}
&\frac{4\gamma\Delta t}{3}(\nabla\cdot(3\phi_{h}^{n+1} - 4\phi_{h}^{n} + \phi_{h}^{n-1}),\nabla\cdot\eta^{n+1})
\\&\leq
\frac{4\gamma\sqrt{d}\Delta t}{3}
\Big(  2\|\nabla\cdot \phi_{h}^{n+1} \| 
+ 2\|\nabla\cdot \phi_{h}^{n} \|  
+  \|\nabla\cdot(\phi_{h}^{n+1} - 2\phi_{h}^{n} + \phi_{h}^{n-1})\|\Big) 
\|\nabla\eta^{n+1}\|
\\&\leq
\frac{2\gamma\Delta t}{3}\|\nabla\cdot(\phi_{h}^{n+1} - 2\phi_{h}^{n} + \phi_{h}^{n-1})\|^{2}
%\\&\quad
+ \frac{\gamma\Delta t}{2}\|\nabla\cdot \phi_{h}^{n+1}\|^{2}
+ \frac{\gamma\Delta t}{2}\|\nabla\cdot \phi_{h}^{n}\|^{2}
+ \frac{70\gamma d \Delta t}{9} \|\nabla\eta^{n+1}\|^{2}.
\end{aligned}
\end{eqnarray}
Then we have 
\begin{eqnarray}\label{pf6.3}
\begin{aligned}
	&\|\phi^{n+1}_{h}\|^{2} - \|\phi^{n}_{h}\|^{2}
	+ \|2\phi^{n+1}_{h}-\phi^{n}_{h}\|^{2}
	- \|2\phi^{n}_{h}-\phi^{n-1}_{h}\|^{2}
	+ \|\phi^{n+1}_{h}-2\phi^{n}_{h}+\phi^{n-1}_{h}\|^{2}
	\\&\quad
	+ \frac{2\gamma\Delta t}{3} 
	\Big( \|\nabla\cdot\phi^{n+1}_{h}\|^{2} - \|\nabla\cdot\phi^{n}_{h}\|^{2}
	+ \|\nabla\cdot(2\phi^{n+1}_{h}-\phi^{n}_{h})\|^{2}
	- \|\nabla\cdot(2\phi^{n}_{h}-\phi^{n-1}_{h})\|^{2}
	\Big)
	\\&\quad 
	+ \frac{\gamma\Delta t}{2} ( 
	\|\nabla\cdot\phi^{n+1}_{h}\|^{2}
	- \|\nabla\cdot\phi^{n}_{h}\|^{2}
	)		
%	\\&\quad 
	+ 6 \|\psi^{n+1}_{h}-\phi^{n+1}_{h}\|^{2}
	+ 2\nu \Delta t \|\nabla\psi^{n+1}_{h}\|^{2}
	+ \gamma\Delta t\|\nabla\cdot\phi^{n+1}_{h}\|^{2}
	\\&\leq
	\frac{C\Delta t}{\nu} \|u^{n+1}\|_{2}^{2} ( \|2\phi^{n}_{h}-\phi^{n-1}_{h}\|^{2} + \|\phi^{n}_{h}\|^{2} + \|\phi^{n-1}_{h}\|^{2} )
	\\&\quad
	+ \frac{C}{\nu}\int^{t^{n+1}}_{t^{n-1}}\|\eta_{t}\|^{2}dt
	+ C (\nu + \gamma d)\Delta t\|\nabla\eta^{n+1}\|^{2}
	\\&\quad
	+ \frac{Cd\Delta t}{\nu}\|p^{n+1} - q_{h}\|^{2}
	+ \frac{C\Delta t^{4}}{\nu}\int^{t^{n+1}}_{t^{n-1}}\|u_{ttt}\|_{-1}^{2}dt
	+ \frac{C\Delta t^{4}}{\nu}\int^{t^{n+1}}_{t^{n-1}}\|\nabla u_{tt}\|^{2}dt
	\\&\quad
	+ \frac{C\Delta t}{\nu} ( \|\nabla\eta^{n}\|^{2} + \|\nabla\eta^{n-1}\|^{2} )
	+ \frac{C\Delta t}{\nu} ( \|\nabla\hat{u}_{h}^{n}\|^{2} + \|\nabla\hat{u}_{h}^{n-1}\|^{2} ) \|\nabla\eta\|_{\infty,0}^{2}.
\end{aligned}
\end{eqnarray}
Sum ($\ref{pf6.3}$) from $n=1$ to $N-1$ to get
\begin{eqnarray}\label{pf6.4}
\begin{aligned}
&\|\phi^{N}_{h}\|^{2} + \|2\phi^{N}_{h}-\phi^{N-1}_{h}\|^{2}
+ \frac{2\gamma\Delta t}{3} 
\Big( \|\nabla\cdot\phi^{N}_{h}\|^{2} + \|\nabla\cdot(2\phi^{N}_{h}-\phi^{N-1}_{h})\|^{2} \Big)
\\&\quad
+ 2\nu \Delta t \sum^{N-1}_{n=1}\|\nabla\psi^{n+1}_{h}\|^{2}
+ \gamma\Delta t \sum^{N-1}_{n=1}\|\nabla\cdot\phi^{n+1}_{h}\|^{2}
\\&\leq
\frac{C\Delta t}{\nu} \sum^{N-1}_{n=1} \|u^{n+1}\|_{2}^{2} ( \|\phi^{n}_{h}\|^{2} + \|2\phi^{n}_{h}-\phi^{n-1}_{h}\|^{2} )
+ C(\gamma d + \nu + \frac{1}{\nu}) |\|\nabla\eta\||_{2,0}^{2}
\\&\quad
+ \frac{C}{\nu}\|\eta_{t}\|_{2,0}^{2}
+ \frac{Cd}{\nu}  |\|p - q_{h}\||_{2,0}^{2}
+ \frac{C\Delta t^{4}}{\nu}\|u_{ttt}\|_{2,-1}^{2}
+ \frac{C\Delta t^{4}}{\nu}\|\nabla u_{tt}\|_{2,0}^{2}
+ \frac{C}{\nu^{2}} \|\nabla\eta\|_{\infty,0}^{2}
\\&\quad
+ \|\phi^{1}_{h}\|^{2} + \|2\phi^{1}_{h}-\phi^{0}_{h}\|^{2}
+ \frac{7\gamma\Delta t}{6}  \|\nabla\cdot\phi^{1}_{h}\|^{2} 
+ \frac{2\gamma\Delta t}{3} \|\nabla\cdot(2\phi^{1}_{h}-\phi^{0}_{h})\|^{2} \Big).
\end{aligned}
\end{eqnarray}
Denote $C^{**}=\frac{C}{\nu}|\|u\||_{2,2}^{2}$.  The result then follows by similar arguments as in Theorem \ref{error}.
\end{proof}

\begin{corollary}
	Under the assumptions of Theorem $\ref{error}$,
	suppose that $(X_{h},Q_{h})$ is given by P2-P1 Taylor-Hood approximation elements ($k=2, m=1$).
	Then, the following estimate holds for \textit{BDF2-mgd}.
	\begin{eqnarray}\label{taylor-hood-error}
	\begin{aligned}
	&\|e_{u}^{N}\|^{2} + \|2e_{u}^{N}-e_{u}^{N-1}\|^{2}
	+ (\frac{2\gamma\Delta t}{3} + \beta)
	\Big( \|\nabla\cdot e_{u}^{N}\|^{2} + \frac{1}{2}\|\nabla\cdot(2e_{u}^{N}-e_{u}^{N-1})\|^{2} \Big)
	\\&\quad
	+ 2\nu \Delta t \sum^{N-1}_{n=1}\|\nabla e_{\hat{u}}^{n+1}\|^{2}
	+ 2\gamma\Delta t \sum^{N-1}_{n=1}\|\nabla\cdot e_{u}^{n+1}\|^{2}
	\\&\leq
	C \Big( h^{6} +  h^{4} + \Delta t \,h^{4} + \Delta t^{4}
	\\&\quad
	+ \|e_{u}^{1}\|^{2} + \|2e_{u}^{1}-e_{u}^{0}\|^{2}
	+ ( \gamma\Delta t + \beta)
	(\|\nabla\cdot e_{u}^{1}\|^{2}
	+ \|\nabla\cdot(2e_{u}^{1}-e_{u}^{0})\|^{2}) \Big).
	\end{aligned}
	\end{eqnarray}
\end{corollary}

\section{Numerical Tests}\label{Numerical Tests}
In this section, we consider three test problems to illustrate the stability, convergence, and effectiveness of \textit{BDF2-mgd}. First, we consider the Taylor-Green benchmark problem to compute convergence rates and test both computational efficiency and pressure-robustness.  We follow with 2D channel flow over a step, where the effect of \textit{BDF2-mgd} on reducing the divergence error is illustrated.  Moreover, it is shown how $ \gamma $ and $ \beta $ influence this effect. Finally, we simulate flow past a cylinder to further present the effectiveness of \textit{BDF2-mgd}.  For all tests, we compare \textit{BDF2-mgd} with BDF2 (Non-Stabilized) and BDF2 with standard grad-div stabilization (Standard Stabilized).  All tests are implemented using FreeFem++ \cite{Hecht}. 

\subsection{Test of Convergence and Pressure Robustness}\label{Taylor-Green}
The Taylor-Green benchmark problem is commonly used to test convergence rates of new algorithms.  As such, we first illustrate convergence rates. The domain is $ [0,1]\times [0,1] $ and final time is $ T=1 $. 
Finite element meshes are generated via Delaunay-Vornoi triangulations with $ m $ points on each side of the boundary. 
The true solution is given by  
\begin{eqnarray*}
\begin{aligned}
&u(x,y,t) = ( -\cos(\omega\pi x)\sin(\omega\pi y), \sin(\omega\pi x)\cos(\omega\pi y) ) \exp( -2\omega^{2}\pi^{2}t/\tau ),\\
&p(x,y,t) = -\frac{1}{4} \Big( \cos(2\omega\pi x) + \cos(2\omega\pi y) \Big) \exp( -4\omega^{2}\pi^{2}t/\tau ).
\end{aligned}
\end{eqnarray*}
Here, $ \omega=1 $, $ \tau=100$, and $ Re=\frac{1}{\nu}=100 $.
The body force $ f $, initial condition, and boundary condition are determined by the true solution.  The grad-div parameters  are set to $ \gamma=1 $, $ \beta=0.2 $.   
The time step is $ \Delta t =  1/m$ where we vary $ m=16, 24, 32, 40 $, and 48 to calculate convergence rates. Table \ref{taylorgreen_rate} presents the results which are consistent with our theoretical analysis. 

To test computational efficiency, we set $ m=32 $ and vary $ \gamma$ and $\beta $.  We compare computational times of Standard Stabilized and \textit{BDF2-mgd}; for $ \gamma = \beta = 0 $, Standard Stabilized is equivalent to Non-Stabilized.  
For Standard Stabilized and $ Step\; 1 $ of \textit{BDF2-mgd}, we use a standard GMRES solver. 
If GMRES fails to converge at a single iterate, we denote the result with an ``F". 
For $ Step\; 2 $ of \textit{BDF2-mgd}, since it leads to an SPD system with same sparse coefficient matrix, at each timestep, we use UMFPACK.
The results are presented in Table \ref{taylorgreen_time}. 
The computing time of Standard Stabilized generally increases as $ \gamma $ and $ \beta $ increase. However, computing time of \textit{BDF2-mgd} is unaffected and therefore increasingly more efficient than Standard Stabilized.  Interestingly, GMRES fails to converge when $ \gamma \simeq 20$ and $ \beta \simeq 0.8 $, which are not very large values.

Lastly, we consider the issue of pressure-robustness.  An advantage of grad-div stabilization is that appropriate selection of the grad-div parameter $ \gamma $ can reduce the effect of the pressure error on the velocity error. Generally, for non-stabilized methods, velocity error estimates result in $ \nu^{-1}\inf\limits_{q_{h}\in Q_{h}}|\|p-p_{h}\||^{2}_{2,0} $ on the right hand side; see, e.g., Theorem 24 on p. 168 of \cite{Layton}. This same term appears for \textit{BDF2-mgd} in Theorems \ref{error} and \ref{error0}. However, for standard grad-div stabilized methods, this term is replaced by $ \gamma^{-1}\inf\limits_{q_{h}\in Q_{h}}|\|p-p_{h}\||^{2}_{2,0} $ in theoretical analyses.
To investigate the sharpness of our results, we vary $ Re $ while fixing $ \Delta t =  1/m = 1/32$ and $ \gamma=1 $, $ \beta=0.2 $.  We compare the velocity and pressure errors of Non-Stabilized, Standard Stabilized, and \textit{BDF2-mgd}.
Results are presented in Table \ref{taylorgreen_Re}. 
It is clear that \textit{velocity errors of Non-Stabilized}, especially for the divergence and gradient, \textit{grow as Re increases}; this is consistent with the corresponding theoretical result.
Alternatively, as $ Re $ is increased, velocity errors of Standard Stabilized and \textit{BDF2-mgd} are consistent with one another and maintain good approximations. This suggests that the effect of $ Re $ appearing in our analysis is not sharp.
This is an open problem, Section \ref{Conclusion}.

\begin{table}
	\begin{tabular}{ c  c  c  c  c  c  c  c	c}
	\hline			
	$m$ & $|\| e_{u} \||_{\infty,0}$ & Rate & $|\| \nabla \cdot e_{u} \||_{\infty,0}$ & Rate & $|\|\nabla \cdot e_{u} \||_{2,0}$ & Rate & $|\| e_{p} \||_{2,0}$ & Rate \\
	\hline
	16 & 2.47E-04 & - & 3.33E-03 & - & 2.82E-03 & - & 1.26E-04 & -\\
	24 & 8.07E-05 & 2.76 & 1.37E-03 & 2.19 & 1.18E-03 & 2.15 & 6.28E-05 & 1.72\\
	32 & 3.54E-05 & 2.86 & 7.21E-04 & 2.24 & 6.24E-04 & 2.21 & 2.87E-05 & 2.73\\
	40 & 1.90E-05 & 2.79 & 5.00E-04 & 1.64 & 4.34E-04 & 1.63 & 2.08E-05 & 1.43\\
	48 & 1.12E-05 & 2.88 & 3.58E-04 & 1.83 & 3.11E-04 & 1.82 & 1.32E-05 & 2.50\\
	\hline  
    \end{tabular}
    \caption{Errors and rates of velocity and pressure for \textit{BDF2-mgd} using the Taylor-Hood element.}
    \label{taylorgreen_rate}
\end{table}

\begin{table}
	\centering
	\begin{tabular}{ c  c  c  c }
		\hline	
		\multicolumn{2}{c} {Parameters} & \multicolumn{2}{c} {Time (s)} \\
		$\beta$ & $\gamma$ & Standard Stabilized & \textit{BDF2-mgd}   \\
		\hline
		0 & 0 & \textbf{17.77} & \textbf{25.09}  \\
		0 & 0.2 & \textbf{30.91} & \textbf{20.10} \\
		0 & 2 & \textbf{55.29} & \textbf{20.45} \\
		0 & 20 & F (\textbf{339.01}) & \textbf{27.99} \\
		0 & 200 & F (\textbf{507.41}) & \textbf{23.88} \\
		0 & 2,000 & F (\textbf{421.66}) & \textbf{17.34} \\
		0 & 20,000 & F (\textbf{27.44}) & \textbf{20.04} \\
		0.01 & 0.2 & 27.79 & 22.28 \\
		0.02 & 0.2 & 32.89 & 22.33 \\
		0.04 & 0.2 & 64.37 & 21.68 \\
		0.08 & 0.2 & 69.31 & 23.97 \\
		0.8 & 0.2 & F & 25.87 \\
		8 & 0.2 & F & 19.53 \\
		80 & 0.2 & F & 21.20 \\
		800 & 0.2 & F & 17.43 \\
		8,000 & 0.2 & F & 18.64 \\
		\hline  
	\end{tabular}
    \caption{Computational time and solver breakdown for Standard and BDF2-mgd with increasing grad-div parameters.}
    \label{taylorgreen_time}
\end{table}

\begin{table}[H]
	\begin{adjustbox}{max width=\textwidth}
	\begin{tabular}{ c  c  c  c  c  c  c  c	c c	cc}
		\hline			
		\multicolumn{1}{c} {Parameter} & \multicolumn{3}{c} {$|\| u_{h}- u \||_{\infty,0}$} & \multicolumn{3}{c} {$|\| \nabla \cdot(u_{h} - u) \||_{2,0}$} & \multicolumn{3}{c} {$|\| \nabla(u_{h} - u) \||_{2,0}$} &\multicolumn{1}{c} {$|\| p_{h} - p \||_{2,0}$}\\
		\hline
		\multirow{2}{*}{$Re$} & 	\multirow{2}{*}{Non-Stabilized}& 
		Standard  & 
        \multirow{2}{*}{\textit{BDF2-mgd}} & 
		\multirow{2}{*}{Non-Stabilized}& 
		Standard & 
		\multirow{2}{*}{\textit{BDF2-mgd}}  & 
		\multirow{2}{*}{Non-Stabilized}& 
		Standard & 
		\multirow{2}{*}{\textit{BDF2-mgd}}  & 
		\multirow{2}{*}{All}\\
		& &Stabilized & 
		& &Stabilized & 
		& &Stabilized & 		
		&&
		\\
		\hline
		1 & 1.26E-03 & 1.26E-03 & 1.26E-03 & 3.61E-05 & 2.43E-05 & 5.33E-05 & 
		4.05E-03 & 4.05E-03 & 4.09E-03 & 
		5.79E-07
		\\
		1e1 & 2.45E-05 & 2.42E-05 & 2.70E-05 & 4.71E-04 & 1.39E-04 & 2.01E-04 & 
		5.48E-04 & 3.67E-04 & 2.07E-03& 
		1.09E-05
		\\
		1e2 & 9.95E-05 & 1.80E-05 & 3.57E-05 & 1.20E-02 & 5.40E-04 & 6.44E-04 & 
		1.27E-02 & 2.80E-03 & 6.65E-03& 
		2.99E-05
		\\
		1e3 & 1.05E-03 & 5.85E-05 & 8.90E-05 & 1.03E-01 & 7.20E-04 & 7.51E-04 & 
		1.31E-01 & 9.04E-03 & 1.15E-02& 
		3.57E-05
		\\
		1e4 & 1.30E-01 & 2.23E-04 & 2.62E-04 & 
		6.27 & 7.61E-04 & 7.78E-04 & 
		7.81 & 2.40E-02 & 3.01E-02& 
		3.64E-05
		\\
		1e5 & 3.57E-01 & 3.99E-04 & 3.50E-04 & 
		16.36 & 7.76E-04 & 7.84E-04 & 
		25.23 & 3.61E-02 & 3.95E-02& 
		3.64E-05
		\\
		1e6 & 4.33E-01 & 4.32E-04 & 3.63E-04 & 
		20.34 & 7.78E-04 & 7.85E-04 & 
		32.44 & 3.84E-02 & 4.09E-02& 
		3.64E-05
		\\
		\hline  
	\end{tabular}
\end{adjustbox}
\caption{Comparison of velocity and pressure errors with increasing $ Re $.}\label{taylorgreen_Re}
\end{table}

\subsection{2D Channel Flow Over a Step}
We now illustrate the effect of $ Step\; 2 $ of \textit{BDF2-mgd} by comparing Non-Stabilized, Standard Stabilized, and \textit{BDF2-mgd} simulations of 2D channel flow over a step \cite{Fragos, John3}. The channel considered here is $ [0,40]\times [0,10] $ with a $ 1\times 1 $ step on the bottom for $x \in [5,6] $. 
A flow with $ \nu=1/600$ passes though this channel from left to right. For boundary conditions, the left inlet and right outlet are given by
\begin{eqnarray*}
\begin{aligned}
&u(0,y,t) = u(40,y,t) = y(10-y)/25,\\
&v(0,y,t) = v(40,y,t) = 0.
\end{aligned}
\end{eqnarray*}
No-slip, $ u=0 $, boundary conditions are imposed elsewhere.
Taylor-Hood elements are used, comprising a mesh with $ 31,089 $ degrees of freedom. The body force $ f = 0 $, final time $ T = 40 $, and time step $ \Delta t =  0.01$. 
The selected grad-div parameters are $ \gamma = 0.1, 0.2, 1 $ and $ \beta = 0, 0.1, 0.2, 1 $. 
$ \|\nabla\cdot u(t^{n})\| $ is computed and plotted in Figure \ref{figure-step1}. 
Also, plots of flow speed and divergence contours, at the final time, with $ \gamma=1, \beta=0 $, are presented in Figure \ref{figure-step2}.

\begin{figure}[htbp]
	\centering
	\includegraphics[width=0.7\textwidth]{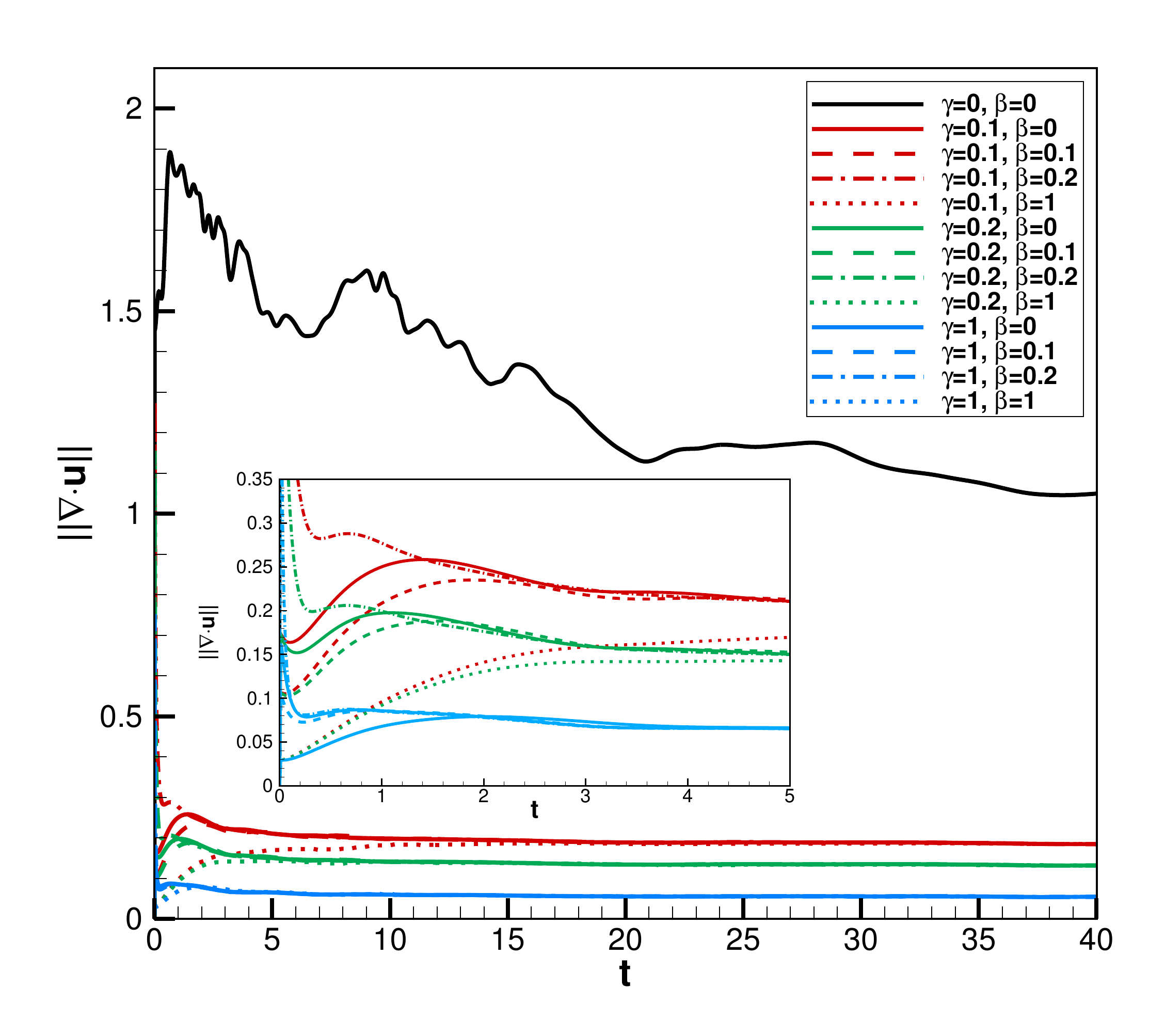}
	\caption{$ \|\nabla\cdot u\| $ vs time for Non-Stabilized and \textit{BDF2-mgd}. }
	\label{figure-step1}
\end{figure}

\begin{figure}[htbp]
	\centering
	\includegraphics[width=0.46\textwidth]{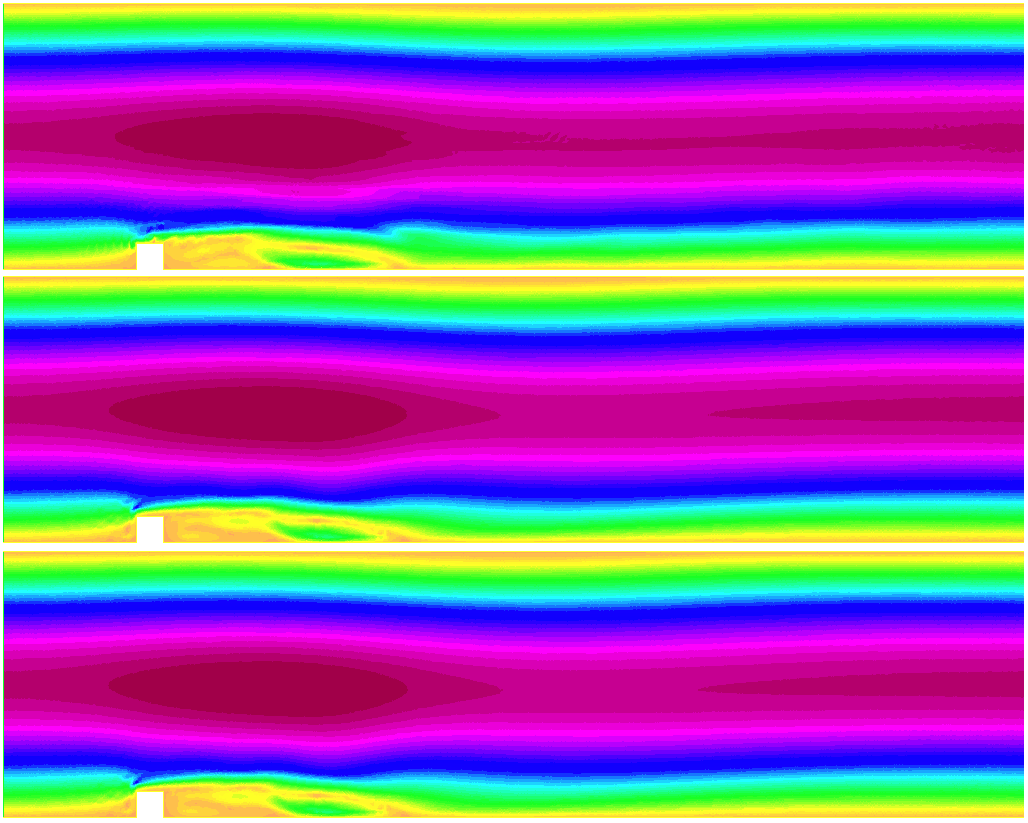}
	\includegraphics[width=0.5\textwidth]{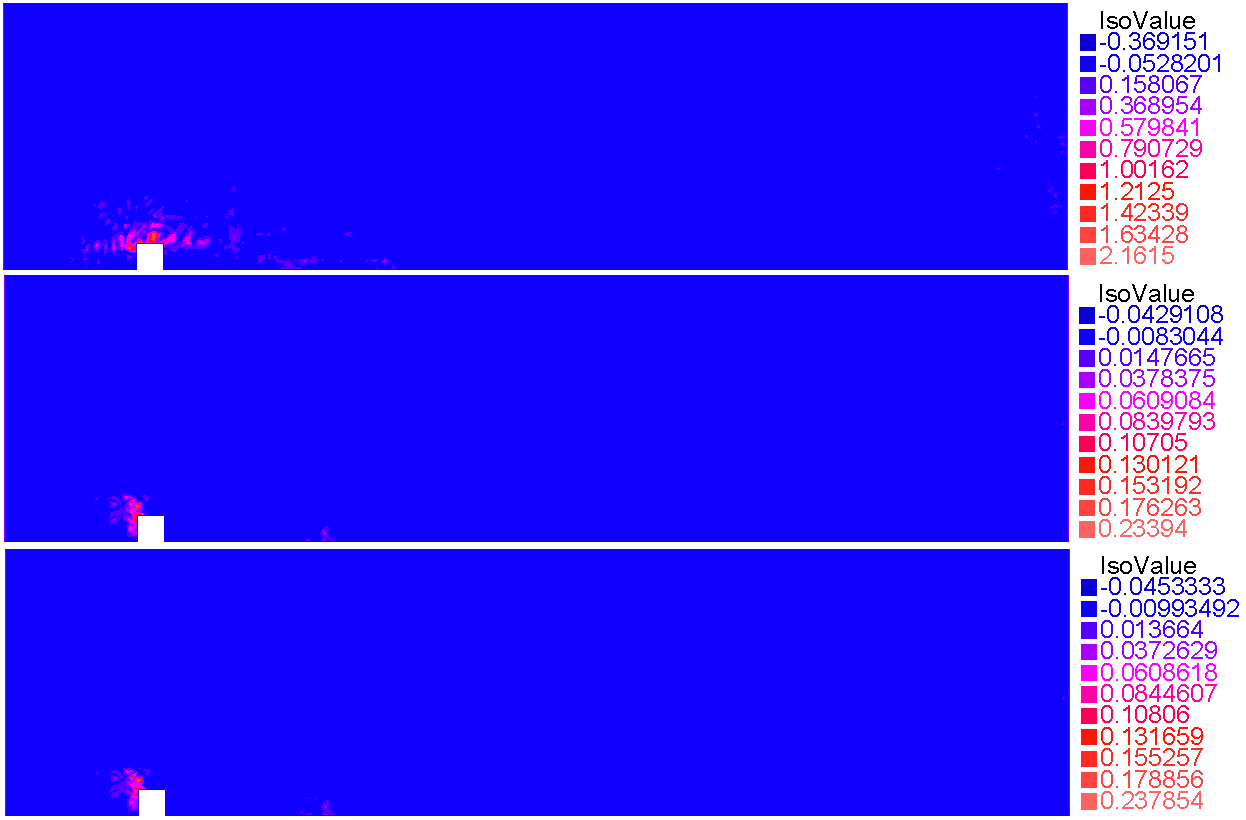}
	\caption{Flow speed and divergence contours at time $ t=40 $ for Non-Stabilized (top), Standard Stabilized (middle) and \textit{BDF2-mgd} (down) with $ \gamma=1, \beta=0 $. }
	\label{figure-step2}
\end{figure}

As shown in Figure \ref{figure-step1}, $ Step\; 2 $ of \textit{BDF2-mgd} greatly reduces the divergence error $\|\nabla\cdot u\|$ compared with Non-Stabilized. Observing the curves of different $ \gamma $ and $ \beta $, it's interesting to find that the value of $ \beta $ determines the minimum divergence error that can be reached in the beginning and the value of $ \gamma $ determines the long-time divergence error.  This is consistent with \cite{Fiordilino}.
In Figure \ref{figure-step2}, we see that results for $ Step\; 2 $ of \textit{BDF2-mgd} are consistent with Standard Stablilzed; both reduce divergence error, especially around the step.

\subsection{2D Channel Flow Past a Cylinder}
In order to further test the effectiveness of \textit{BDF2-mgd}, we consider channel flow past a cylinder \cite{Schafer}. Like the Taylor-Green benchmark, this is a common test problem for new algorithms.
The channel domain is $ [0,2.2]\times [0,0.41] $ with a cylinder of diameter $ 0.1 $ within. The center of the cylinder is $ (0.2, 0.2) $. A flow with $ \nu=0.001, \rho=1 $ passes though this channel from left to right. No body forces are present, $ f = 0 $.
Left in-flow and right out-flow boundaries are given by 
\begin{eqnarray*}
\begin{aligned}
&u(0,y,t) = u(2.2,y,t) = \frac{6y(0.41-y)}{0.41^{2}}\sin(\frac{\pi t}{8}),\\
&v(0,y,t) = v(2.2,y,t) = 0.
\end{aligned}
\end{eqnarray*}
The no-slip boundary condition is prescribed elsewhere.

We use Taylor-Hood elements on a mesh with $ 41,042 $ degree of freedom and final time $ T = 8 $. The time step is $ \Delta t =  0.001$. The grad-div parameters are set to $ \gamma = 5\nu $ and $ \beta = 0$. 
Drag $ c _{d}(t) $ and lift $ c_{l}(t) $ coefficients are calculated; maximum values are presented in Table \ref{table=cylinder}. The pressure difference between the front and back of the cylinder ($ \Delta p(t)=p(0.15,0.2,t)-p(0.25,0.2,t) $) and both the $L^2(0,T;L^{2}(\Omega))$ and $L^{\infty}(0,T;L^{2}(\Omega))$ norms of the velocity divergence are also tabulated in Table \ref{table=cylinder}. 
Furthermore, Figure \ref{figure-cylinder} shows velocity speed and vectors for \textit{BDF2-mgd} at times $ t=4, 6, 7, 8$, which are consistent with that in \cite{Bowers, Fiordilino, John4, Linke}.

In Table \ref{table=cylinder}, we see that grad-div stabilization effectively reduces the divergence error, as expected. This results in improved accuracy of Standard Stabilized and \textit{BDF2-mgd} over the Non-Stabilized solution.  In particular, both stabilized algorithms produce accurate lift coefficients and smaller divergence errors. 

\begin{table}[htbp]\normalsize
	\centering
	\begin{tabular}{ c	c	c	c	c	c}
		\hline			
		Method & $c^{max}_{d}$ & $c^{max}_{l}$ & $\Delta p^{N}_{h}$ & $|\|\nabla \cdot u_{h}|\|_{2,0}$ & $\| \nabla \cdot u^{N}_{h}\|$\\
		\hline
		Non-Stabilized & 2.950 & \textbf{0.441} & -0.1084 & \textbf{1.967} & \textbf{0.186} \\
		Standard Stabilized & 2.950 & 0.477 & -0.1115 & 0.859 & 0.072\\
		\textit{BDF2-mgd} & 2.950 & 0.475 & -0.1115 & 0.906 & 0.074\\
		%Modular + Lagged & 2.950 & 0.486 & \textbf{-0.1037} & 0.029 \\
		\hline  
	\end{tabular}
	\caption{Maximum lift, drag coefficients, pressure drop, and divergence quantities for flow past a cylinder.}\label{table=cylinder}
\end{table}

\begin{figure}[htbp]
	\centering
	\includegraphics[width=0.48\textwidth]{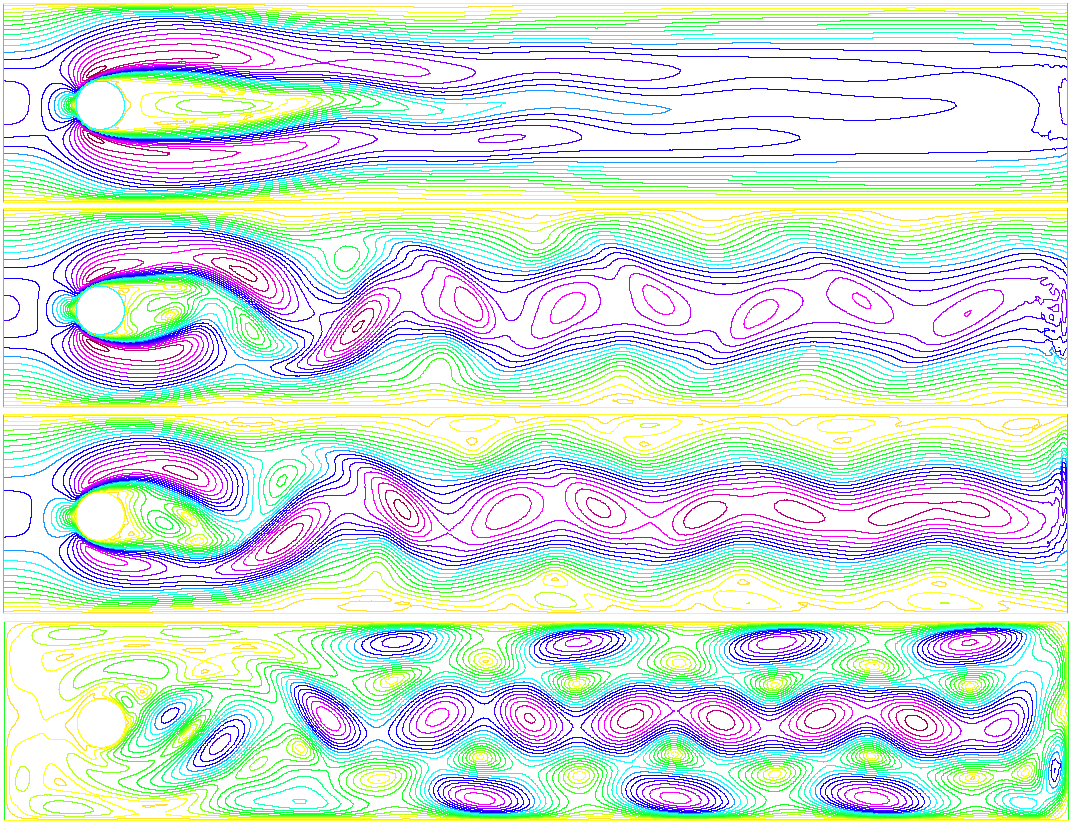}
	\includegraphics[width=0.48\textwidth]{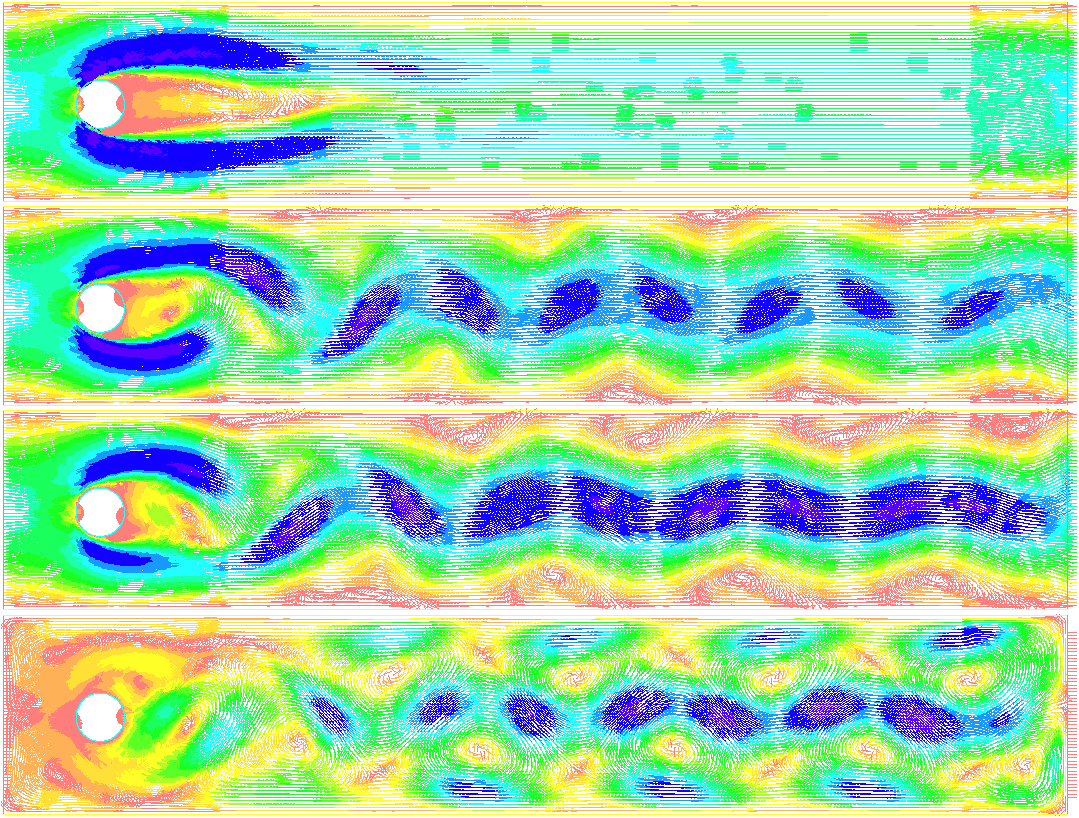}
	\caption{Flow speed and vectors for flow past a cylinder at times  t=4, 6, 7, and 8. }
	\label{figure-cylinder}
\end{figure}

\section{Conclusion}\label{Conclusion}
We developed a BDF2 time-discrete, modular grad-div stabilization algorithm (\textit{BDF2-mgd}) for the time dependent Navier-Stokes equations.
Compared with methods implementing standard grad-div stabilization, our algorithm produces consistent numerical approximations while avoiding solver breakdown for large grad-div parameters.
We prove that this algorithm is unconditionally, nonlinearly, energy stable and second-order accurate in time.
Numerical tests illustrate the theoretical results and computational efficiency.   

To impose discrete versions of $-\beta \nabla \nabla \cdot u_t - \gamma \nabla \nabla \cdot u$, modular grad-div requires a solve of the form $\big(\frac{1}{\Delta t}I + (\frac{\beta}{\delta t} + \gamma) G\big)u = RHS$, where $G$ is the symmetric positive semi-definite grad-div matrix.  For constant $\Delta t$, efficiency increases can exploit the fact that the matrix is fixed.  For variable timestep and $\beta =0$, the matrix is a variable shift of $G$ and efficient algorithms exist exploiting this structure.  Important next steps include investigating, analytically, the $\nu$ dependence of $ \nu^{-1}\inf\limits_{q_{h}\in Q_{h}}|\|p-p_{h}\||^{2}_{2,0} $ in Theorems \ref{error} and \ref{error0}, extending these results to alternative numerical methods, and including sparse, effective variants of grad-div stabilization.

\end{document}